\documentclass[11pt]{article}
\usepackage{graphicx,wrapfig,lipsum}
\usepackage{amsmath,amssymb,graphicx} 
\usepackage{tcolorbox}
\usepackage{amsfonts}

\usepackage{mathtools}
\usepackage{extarrows}
\usepackage[margin=1.0in]{geometry} 
\usepackage{enumerate}
\usepackage[toc]{appendix}
\usepackage{hyperref}
\usepackage{fullpage}
\usepackage{tikz}
\usepackage{tcolorbox}
\usepackage{bbm}
\usepackage{comment}
\newtheorem{definition}{Definition}[section]
\newtheorem{theorem}{Theorem}

\newtheorem{lemma}{Lemma}
\newtheorem{remark}{Remark}
\usepackage{comment}
\usepackage{enumitem}
\usepackage{stackengine}
\usepackage[utf8]{inputenc}
\usepackage{xcolor}
\usepackage{graphicx}
\graphicspath{ {./images/} }
\allowdisplaybreaks

\newenvironment{proof}{\paragraph{Proof:}}{\hfill$\square$}
\allowdisplaybreaks
\def\es{\varepsilon}
\usepackage{color}
\newcommand{\xrightharpoonupdown}[2][-12pt]{%
\stackunder[#1]{$\xrightharpoonup{\text{#2}}$}{}%
}%

\begin{document}
	
	
	\title{Homogenization of a reaction-diffusion problem with large nonlinear drift and Robin boundary data}

\author{Vishnu Raveendran$^{a,*}$,  Ida de Bonis$^{b}$, Emilio N.M. Cirillo$^c$, Adrian Muntean$^a$ \\
$^a$ Department of Mathematics and Computer Science, Karlstad University, Sweden\\
$^b$ Dipartimento di Pianificazione, Design, Tecnologia dell'Architettura, \\ Sapienza Universit{\`a} di Roma, Italy\\
$^c$ Dipartimento di Scienze di Base e Applicate per l’Ingegneria, \\ Sapienza Universit{\`a} di Roma, Italy\\
* vishnu.raveendran@kau.se}

	\date{\today} 
	\maketitle
	
	\begin{abstract}\label{abstract}
	\noindent We study the periodic homogenization of a reaction-diffusion problem with large nonlinear drift and Robin boundary condition posed in an unbounded perforated domain. The nonlinear problem is associated with the hydrodynamic limit of a totally asymmetric simple exclusion process (TASEP) governing a population of interacting particles crossing a domain with obstacle. We are interested in deriving rigorously the upscaled model equations and the corresponding effective coefficients for the case when the microscopic dynamics are linked to a particular choice of characteristic length and time scales that lead to an exploding nonlinear drift. The main mathematical difficulty lies in proving the two-scale compactness and strong convergence results needed for the passage to the homogenization limit. To cope with the situation, we use the concept of two-scale compactness with drift, which is similar to the more classical two-scale compactness result but it is defined now in moving coordinates. We provide as well a  strong convergence result for the corrector function, starting this way the search for the order of the convergence rate of the homogenization process for our target nonlinear drift problem.  
	\end{abstract}
\maketitle

\section{Introduction}\label{In}

Our interest lies in performing the periodic homogenization limit for  a reaction-diffusion problem with large nonlinear drift
and Robin boundary condition posed in an unbounded perforated domain; see the structure of our target equations \eqref{meq}--\eqref{meqic}. The particular nonlinear drift problem we have in mind is associated with the hydrodynamic limit of a totally asymmetric simple exclusion process
(TASEP) governing a population of interacting particles crossing a domain with obstacle \cite{CIRILLO2016436}. In this framework, we aim to
 derive rigorously the upscaled model equations and the corresponding effective coefficients for the case when the microscopic dynamics are linked to a particular choice of
characteristic length and time scales that lead to an exploding nonlinear drift. Ideally, we wish to get as well some insights into the structure of the corrector for the homogenization procedure. 

The question of exploding drifts is a bit unusual in the context of homogenization asymptotics as only particular scalings leading to blow-up can be handled rigorously. Hence, besides a couple or relatively recent papers [we are going to mention a few of them briefly in the following], there are not many contributions in the area. The celebrated concept of two-scale convergence with drift meant to cope with at least one particular exploding scaling has been introduced in \cite{marusik2005} and later developed in  \cite{allaire2016,allaire2010homogenization,hutridurga,hutridur} very much motivated by attempts to understand to so-called turbulent diffusion (see e.g. \cite{mclaughlin1985convection,holding2017convergence} where one speaks about convective microstructures).
 Our own motivation is aligned with statistical mechanics-motivated approaches  to modeling reactive transport in porous media \cite{cirillo2020upscaling}. For phenomenological approaches to describing the physics of porous media, we refer the reader for instance  to \cite{bear2012phenomenological,vafai2015handbook}. Closely related large-drift upscaling questions appear e.g. in the design of filters to improve catalysis \cite{ILIEV2020103779} and in the control of microstructure for an efficient  smoldering combustion in solid fuels \cite{ijioma2019fast}. The classical Taylor dispersion topic belongs to this context as well, compare \cite{ALLAIRE20102292,allaire2016,piatnitski2020homogenization}.  Another class of periodic homogenization problems in the same framework arise from suitably scaled SDEs descriptions with exploding "volatilities" that lead to Fokker-Plank type counterparts with correspondingly exploding drifts (cf. e.g. \cite{hairer2011,zhang2023quantitative}). Other relevant work related to large drift homogenization problems posed for different geometries can be found e.g. in \cite{ALLAIRE2012300,allaire2012homogenization}.

This work is organized as follows: In section \ref{MicM}, we introduced our microscopic model and the geometry of the heterogeneous domain. Then we list the structural assumptions that we rely on to study the homogenization limit of our  problem. In section \ref{Wellpos}, we show the existence and uniqueness of strong solutions to the target microscopic problem. The difficulty is here twofold: the type of nonlinearity and the unboundedness of the domain. We first study the model equations posed on o a bounded domain and then treat the case of the unbounded domain, which is where the microscopic problem needs to be formulated. To this end, we are using a suitable extension of the concept of solution, a comparison principle, jointly  with a suitable monotonicity argument. 
 We conclude this section by showing $\es$-independent energy estimates and the positivity of the solution. These estimates are key ingredients for the passage to the homogenization limit.  In section \ref{upscalemodel}, we prove our main result which is the rigorous upscaling of the microscopic problem posed in an unbounded periodically perforated domain. The next steps follow the path of the periodic homogenization technique. We first define an extension operator which preserve the energy estimates from the original problem. Finally, we employ the two-scale convergence with drift and related compactness results together with strong convergence-type arguments  in moving coordinates. Using such results, we derive the upscaled model which is a nonlinear reaction-dispersion equation coupled with an elliptic cell problem. In section \ref{searchforcor}, we study the difference of solution to micro-problem to solution to macro-problem in $H^1$ norm with the help of a corrector function. A couple of brief remarks and some ideas for potential future work are the subject of the closing section.
 
\section{The microscopic model}\label{MicM}
\begin{figure}[b]
		\begin{center}
			\begin{tikzpicture}
			\draw (0,0)  to (6,0) to (6,3.75) to (0,3.75) to (0,0);
			\draw (1,0) to (1,3.75)  ;
			\draw (2,0) to (2,3.75)  ;
			\draw (3,0) to (3,3.75)  ;
			\draw (4,0) to (4,3.75)  ;
			\draw (5,0) to (5,3.75)  ;
			\draw (0,.75) to (6,.75);
			\draw(0,1.5) to (6,1.5);
			\draw(0,2.25) to (6,2.25);
			\draw(0,3) to (6,3);
			\draw(0,3.75) to (6,3.75);

	\filldraw [black] plot [smooth cycle] coordinates {(0.5,.32) (0.4,.12) (.7,.22) (.7,.42)(0.4,.52)};
	\filldraw [black] plot [smooth cycle] coordinates {(1.5,.32) (1.4,.12) (1.7,.22) (1.7,.42)(1.4,.52)};
	\filldraw [black] plot [smooth cycle] coordinates {(2.5,.32) (2.4,.12) (2.7,.22) (2.7,.42)(2.4,.52)};
	\filldraw [black] plot [smooth cycle] coordinates {(3.5,.32) (3.4,.12) (3.7,.22) (3.7,.42)(3.4,.52)};
	\filldraw [black] plot [smooth cycle] coordinates {(4.5,.32) (4.4,.12) (4.7,.22) (4.7,.42)(4.4,.52)};
	\filldraw [black] plot [smooth cycle] coordinates {(5.5,.32) (5.4,.12) (5.7,.22) (5.7,.42)(5.4,.52)};

			\filldraw [black] plot [smooth cycle] coordinates {(0.5,1.32-.25) (0.4,1.12-.25) (.7,1.22-.25) (.7,1.42-.25)(0.4,1.52-.25)};
	\filldraw [black] plot [smooth cycle] coordinates {(1.5,.32+.75) (1.4,.12+.75) (1.7,.22+.75) (1.7,.42+.75)(1.4,.52+.75)};
	\filldraw [black] plot [smooth cycle] coordinates {(2.5,.32+.75) (2.4,.12+.75) (2.7,.22+.75) (2.7,.42+.75)(2.4,.52+.75)};
	\filldraw [black] plot [smooth cycle] coordinates {(3.5,.32+.75) (3.4,.12+.75) (3.7,.22+.75) (3.7,.42+.75)(3.4,.52+.75)};
	\filldraw [black] plot [smooth cycle] coordinates {(4.5,.32+.75) (4.4,.12+.75) (4.7,.22+.75) (4.7,.42+.75)(4.4,.52+.75)};
	\filldraw [black] plot [smooth cycle] coordinates {(5.5,.32+.75) (5.4,.12+.75) (5.7,.22+.75) (5.7,.42+.75)(5.4,.52+.75)};	
			
			\filldraw [black] plot [smooth cycle] coordinates {(0.5,.32+1.5) (0.4,.12+1.5) (.7,.22+1.5) (.7,.42+1.5)(0.4,.52+1.5)};
	\filldraw [black] plot [smooth cycle] coordinates {(1.5,.32+1.5) (1.4,.12+1.5) (1.7,.22+1.5) (1.7,.42+1.5)(1.4,.52+1.5)};
	\filldraw [black] plot [smooth cycle] coordinates {(2.5,.32+1.5) (2.4,.12+1.5) (2.7,.22+1.5) (2.7,.42+1.5)(2.4,.52+1.5)};
	\filldraw [black] plot [smooth cycle] coordinates {(3.5,.32+1.5) (3.4,.12+1.5) (3.7,.22+1.5) (3.7,.42+1.5)(3.4,.52+1.5)};
	\filldraw [black] plot [smooth cycle] coordinates {(4.5,.32+1.5) (4.4,.12+1.5) (4.7,.22+1.5) (4.7,.42+1.5)(4.4,.52+1.5)};
	\filldraw [black] plot [smooth cycle] coordinates {(5.5,.32+1.5) (5.4,.12+1.5) (5.7,.22+1.5) (5.7,.42+1.5)(5.4,.52+1.5)};
		\filldraw [black] plot [smooth cycle] coordinates {(0.5,.32+2.25) (0.4,.12+2.25) (.7,.22+2.25) (.7,.42+2.25)(0.4,.52+2.25)};
	\filldraw [black] plot [smooth cycle] coordinates {(1.5,.32+2.25) (1.4,.12+2.25) (1.7,.22+2.25) (1.7,.42+2.25)(1.4,.52+2.25)};
	\filldraw [black] plot [smooth cycle] coordinates {(2.5,.32+2.25) (2.4,.12+2.25) (2.7,.22+2.25) (2.7,.42+2.25)(2.4,.52+2.25)};
	\filldraw [black] plot [smooth cycle] coordinates {(3.5,.32+2.25) (3.4,.12+2.25) (3.7,.22+2.25) (3.7,.42+2.25)(3.4,.52+2.25)};
	\filldraw [black] plot [smooth cycle] coordinates {(4.5,.32+2.25) (4.4,.12+2.25) (4.7,.22+2.25) (4.7,.42+2.25)(4.4,.52+2.25)};
	\filldraw [black] plot [smooth cycle] coordinates {(5.5,.32+2.25) (5.4,.12+2.25) (5.7,.22+2.25) (5.7,.42+2.25)(5.4,.52+2.25)};
		
	\filldraw [black] plot [smooth cycle] coordinates {(0.5,.32+3) (0.4,.12+3) (.7,.22+3) (.7,.42+3)(0.4,.52+3)};
	\filldraw [black] plot [smooth cycle] coordinates {(1.5,.32+3) (1.4,.12+3) (1.7,.22+3) (1.7,.42+3)(1.4,.52+3)};
	\filldraw [black] plot [smooth cycle] coordinates {(2.5,.32+3) (2.4,.12+3) (2.7,.22+3) (2.7,.42+3)(2.4,.52+3)};
	\filldraw [black] plot [smooth cycle] coordinates {(3.5,.32+3) (3.4,.12+3) (3.7,.22+3) (3.7,.42+3)(3.4,.52+3)};
	\filldraw [black] plot [smooth cycle] coordinates {(4.5,.32+3) (4.4,.12+3) (4.7,.22+3) (4.7,.42+3)(4.4,.52+3)};
	\filldraw [black] plot [smooth cycle] coordinates {(5.5,.32+3) (5.4,.12+3) (5.7,.22+3) (5.7,.42+3)(5.4,.52+3)};
         	\draw [<->] (6.1,0) to (6.1,3.75);
         	\draw[->](7,3.5)node[anchor=west] {{\scriptsize $\Omega_{0}^{\varepsilon}$}} to (5.5,3.25);
         	\draw[->](7,3.5) to (5.5,2.70);
         	\draw[->](7,3.5)node[anchor=west] {{\scriptsize $\Omega_{0}^{\varepsilon}$}} to (5.5,3.25);
         	\draw[->](7,3.5) to (5.5,1.95);
         	\draw[->](7,0.5)node[anchor=west] {{\scriptsize $\es Z$}} to (5.85,0.1);
         		\draw[->](7,2.5)node[anchor=west] {{\scriptsize $\Omega_{\varepsilon}$}} to (5.5,1.4);
         		\draw[->](7,2.5)node[anchor=west] {{\scriptsize $\Omega_{\varepsilon}$}} to (4.5,1.4);
         		\draw[->](7,2.5)node[anchor=west] {{\scriptsize $\Omega_{\varepsilon}$}} to (5.5,.6);
			\end{tikzpicture}
			\caption{Schematic representation of the {periodic} geometry corresponding to the {target} microscopic model.}
			\label{fig2}
		\end{center}
	\end{figure}
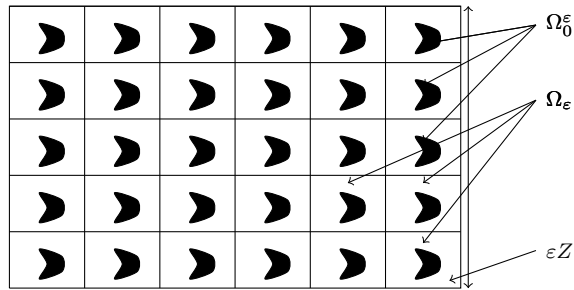

Let  $Y\subset \mathbb{R}^2$ be a unit square in $\mathbb{R}^2$. We define the standard cell $Z$ as $Y$ having as inclusion an impenetrable compact object $Z_0$ called obstacle that is placed inside $Y$ (i.e. $Z:=Y\backslash Z_0$). We assume $\partial Z_0$ has $C^2$ boundary regularity and $\partial Y\cap \partial Z_0=\emptyset .$ We denote $\partial Z_0 $ as $\Gamma_N$ To give our geometry a porous media description,  we define the pore skeleton to be 
\begin{equation*}
    \Omega_{0}^\es:=\bigcup_{(k_1,k_2)\in \mathbb{N}\times \mathbb{N}}\{\es(Z_0+\Sigma_{i=1}^2 k_i e_i)\},
\end{equation*}
where $ \es>0$ and  $\{e_1,e_2\}$ is the orthonormal  basis of $\mathbb{R}^2$. We can now describe the (open) total pore space and its total internal boundary via
\begin{equation}\label{ome}
    \Omega^\es:=\mathbb{R}^2\setminus\Omega_0^\es,
\end{equation}
and respectively, 
\begin{equation*}
    \Gamma_{N}^\es:= \bigcup_{(k_1,k_2)\in \mathbb{N}\times \mathbb{N}} \{\es(\Gamma_N+\Sigma_{i=1}^2 k_i e_i)\}.
\end{equation*}
We denote by $n_\es$, and respectively, $n_y $ the unit normal vectors across the interfaces $ \Gamma_{N}^\varepsilon$, and respectively, $\Gamma_{N}$; they are directed outward with respect to $\partial\Omega_\es$.
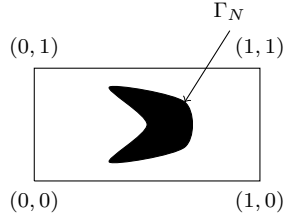
\begin{figure}[ht]
	\begin{center}
		\begin{tikzpicture}
		\draw (0,0) node [anchor=north] {{\scriptsize $(0,0)$}} to (3,0) node [anchor=north] {{\scriptsize $(1,0)$}}  to (3,1.5)node [anchor=south] {{\scriptsize $(1,1)$}} to (0,1.5)node [anchor=south] {{\scriptsize $(0,1)$}} to (0,0);
			\filldraw [black] plot [smooth cycle] coordinates {(1.5,.75) (01,.25) (2,.45) (2,1.05)(1,1.25)};
			\draw[<-](2,1.05) to (2.6,2) node[anchor=south] {{\scriptsize $\Gamma_{N}$}};
		\end{tikzpicture}
		\caption{Standard cell $Z$ exhibiting an obstacle $Y_0$ placed in the center.  
		}\label{scell}
	\end{center}
\end{figure}
\\
  \par  As target system, we consider the following reaction-diffusion-convection problem
   \begin{tcolorbox}
\begin{align}
    \frac{\partial u^{\varepsilon}}{\partial t} +\mathrm{div}(-D^{\varepsilon}\nabla u^{\varepsilon}+ \frac{1}{\es}B^{\varepsilon}P(u^{\varepsilon}))&=f^{\varepsilon}  \,\,&\mbox{on}& \,\,&  \Omega^{\varepsilon} \times (0,T),\label{meq}\\
    (-D^{\varepsilon}\nabla u^{\varepsilon}+ \frac{1}{\es}B^{\varepsilon}P(u^{\varepsilon}))\cdot n_{\es}&=\es g_N(u^\es) &\mbox{on}&   & \,\Gamma_N^\es \times (0,T),\label{meqbcn}\\
     u^\es(0)&=g &\mbox{in}&   &\, \overline{\Omega}^\es\label{meqic}.
\end{align}
\end{tcolorbox}
Here $f^\es:\Omega_\es\rightarrow \mathbb{R}$ and $g_N:\mathbb{R}\rightarrow\mathbb{R}$  are given functions, $D^\es(x_1,x_2):=D(x_1/\es, x_2/\es)$ for $(x_1,x_2)\in \Omega_\es$, where $D$ is a $2\times 2$ matrix and $Z$--periodic defined in the standard unit cell $Z$, $B^\es(x_1,x_2):=B(x_1/\es, x_2/\es)$ where $B$ is a $2\times 1$ vector with positive entries and $Z$--periodic, $C^\es(x_1,x_2):=C(x_1/\es, x_2/\es)$ where $C(\cdot)$ is $Z$ periodic function. 
 What concerns the nonlinear drift $P(\cdot):\mathbb{R}\rightarrow\mathbb{R}$, we set $P(u^\es)$ as  
 \begin{equation}\label{p3}
    P(u^\es)=u^\es(1-C^\es u^\es),
\end{equation} with 
\begin{equation}\label{bc}
     \int_Z BC dy=0.
\end{equation}
If $C^\es=1$, then the structure of the nonlinear drift, i.e. $B^{\varepsilon}P(u^{\varepsilon})$, is precisely what one gets as mean field limit for a suitable TASEP process (cf. \cite{CIRILLO2016436}). It is worth noting that the homogenization question has been posed already for such kind of problem (see \cite{cirillo2020upscaling,raveendran21,raveendran2022upscaling}). The novelty here is the treatment of the exploding  scaling on nonlinear drift.
\vspace{.25cm}
\subsection{Structural assumptions}\label{assumption}
We consider the following restrictions on data and model parameters. We summarize them in the list of assumptions \ref{A1}--\ref{A5}, viz.
 \begin{enumerate}[label=({A}{{\arabic*}})]
 
 \item 
 For all $\eta\in \mathbb{R}^2$ there exists $\theta, \Tilde{\theta}>0$  such that 
 \begin{equation*}\label{}
 \theta \|\eta \|^{2} \leq
 D\eta\cdot \eta\leq \Tilde{\theta}\|\eta \|^{2}, 
 \end{equation*}
 \label{A1}
 \begin{align*}
     D^\es\in C_{\#}^{2,\beta}(Z)^{2\times 2},
 \end{align*}
 where $\beta\in (0,1)$.
\item 
 $B\in C_{\#}^{1,\beta}(Z)^{2}$, $C\in C_{\#}^{1,\beta}(Z)$ satisfies
\begin{equation*}\label{}
\begin{cases}
 	 \mathrm{div}B=0\hspace{.4cm}\mbox{in} \hspace{.4cm}  (0,T)\times Z\\
 	 \mathrm{div}(BC)=0\hspace{.4cm}\mbox{in} \hspace{.4cm}  (0,T)\times Z\\
 	  B\cdot n_y =0 \hspace{.4cm}\mbox{on} \hspace{.4cm}  (0,T)\times \Gamma_N\label{boundaryb}.
 	 \end{cases}
 	 \end{equation*}.
 	 \label{A2}
 	 
  \item  $f^\es :\mathbb{R}^2\rightarrow \mathbb{R}^+$ such that $f^\es \in C_{c}^{2}(\mathbb{R}^2)$ and
  $f^\es\xrightharpoonupdown{$2-drift(B^*)$}f$, where the convergence is defined in the sense of Definition \ref{def2scd}. 

 
 \label{A3}
\item $g_N\in C^{1}(\mathbb{R})$ satisfies
\begin{align}
     -g_N(x)x&<0 \;\;\mbox{for all} \;\;x\neq 0,\label{gnc2}\\
     g_N(x)&\geq g_N(y) \;\;\mbox{} \;\;x\geq y\label{gnc3}.
\end{align}
\label{A4}
\item $g:\mathbb{R}^2\rightarrow \mathbb{R}^+$ such that
\begin{equation*}
    g\in  C_c^\infty(\mathbb{R}^2).
\end{equation*}\label{A5}
\end{enumerate}
 \vspace{.5cm}

Assumptions   \ref{A1}--\ref{A5} are  technical, but they have all a natural physical explanation. \ref{A1} is linked to the choice of non-degenerating diffusion process in the underlying stochastic description, \ref{A2} points out the incompressibility of the flow and how the flow behaves along the boundary of the perforations (the microscopic obstacles), and \ref{A3}--\ref{A5} are simple structural choices for the volume and boundary production terms. Both \ref{A1} and \ref{A2} are essential for the success of our work, while \ref{A3}--\ref{A5} can be replaced by other suitable options.   Note that the relatively high regularity stated  in \ref{A1}--\ref{A5} is primarily needed to ensure the well-posedness of the microscopic problem. To reach the  homogenization limit and to guarantee the strong convergence of the corrector we only need minimal regularity on the data.

\section{Well-posedness of the microscopic problem}\label{Wellpos}

Let us introduce our concept of solution.

\begin{definition}\label{weakformul}
A weak solution to problem \eqref{meq}--\eqref{meqic} is a function $u^\es\in L^2(0,T;H^1(\Omega^\es))\cap H^1(0,T;L^2(\Omega^\es)) $  
satisfying the identity 
\begin{align}
      \int_{\Omega^\es} \partial_t u^\es \phi dx+\int_{\Omega^\es} D^\es \nabla u^\es \nabla \phi dx -\dfrac{1}{\es} \int_{\Omega^\es} B^\es u^\es(1-C^\es u^\es) \nabla \phi dx=  \int_{\Omega^\es} f^\es \phi dx- \es\int_{{\Gamma_N^\es}} g_N(u^\es) \phi d\sigma\label{wf}
\end{align}
for all $\phi\in H^1(\Omega^\es)$ and a.e. $t\in (0,T)$ with initial condition $u(0)=g$.
\end{definition}
Before proceeding to any asymptotic study of type $\es\to 0$, we must ensure first that this concept of solution is suitable for the problem at hand. 

\vspace{.5cm}
\subsection{Analysis for bounded domains, extension arguments, $\es$-independent bounds}
In this section, {we are concerned with guaranteeing the weak solvability of} the microscopic problem \eqref{meq}--\eqref{meqic}, i.e. we look for solutions in the sense of Definition \ref{weakformul}. {Relying on fundamental results from \cite{ladyzhenskaia1988linear}, we first show the existence of a strong solution for the same model equations while they are formulated in a bounded domain.} 
{As next step,} by using a comparison principle {combined with} a monotone convergence argument, we show that a particular {sequence of extensions of the solution to the bounded domain problem} converges to the weak solution of our {original} problem \eqref{meq}--\eqref{meqic} posed on an unbounded domain.

\begin{theorem}\label{Lad}
Consider the following nonlinear reaction diffusion convection equation
\begin{align}
    u_t-a_{ij}(t,x,u)u_{x_ix_j}+b(t,x,u,u_{x})&=0 \,\;\;\;\mbox{on}\;\;\; (0,T)\times\Omega\label{Lp1}\\
    a_{ij}(t,x,u)u_{x_i}\cdot n&= \psi \,\;\;\;\mbox{on}\;\;\; (0,T)\times\partial\Omega\label{Lp2}\\
    u(0,x)=\psi _0(x,u)
\end{align}
with the assumptions 
\begin{itemize}
    \item [(P1)] \begin{align*}
        \nu \xi^2 \leq a_{ij}(t,x,u)\xi_{i}\xi_{j}\leq \mu \xi^2.
    \end{align*}
    \item [(P2)] If $|y|\leq M$ for some constant {$M>0$}, $\frac{\partial a_{ij}(t,x,y)}{\partial y}, \frac{\partial a_{ij}(t,x,y)}{\partial x}, \psi, \frac{\partial \psi(t,x,y)}{\partial x},\frac{\partial \psi(t,x,y)}{\partial y}, \frac{\partial^2 a_{ij}(t,x,y)}{\partial^2y}, $\\   $\frac{\partial^2 a_{ij}(t,x,y)}{\partial y\partial t}, \frac{\partial^2 a_{ij}(t,x,y)}{\partial x\partial t} $
        are uniformly bounded in their respective domain.
   \item [(P3)] For fixed $x,t,y$ and arbitrary $p$, there exist $\mu>0$
   \begin{align*}
       |b(t,x,y,p)|\leq \mu (1+p^2).
   \end{align*}
   \item [(P4)] There exist  $C_0, C_1,C_2>0$ such that for $(t,x)\in (0,T)\times\Omega$,  $b(t,x,y,p)$ satisfies the condition 
   \begin{align*}
       -yb(t,x,y,0)\leq C_0y^2+C_1,\\
       |b_p|(1+|p|)+|b_y|+|b_t|\leq C_2(1+p^2). 
   \end{align*}
   \item [(P5)] {There exist  $C_0, C_1>0$ such that for $(t,x)\in (0,T)\times\partial\Omega$,  $\psi(t,x,y)$ satisfies the condition 
   \begin{align*}
       -y\psi(t,x,y)< 0 \;\;\;\mbox{for all}\;\;\; (t,x)\in 
(0,T) \times\partial \Omega\;\mbox{and}\; |y|>0.
       \end{align*}}
     \item [(P6)] For $|y|\leq M $, where $M$ is a constant, $a_{ij}(t,x,y), b(t,x,y,p) $ and $\psi(t,x,y)$ are continuous in their arguments.
     \item [(P7)] For $|y|,|p|\leq M $ ${a_{ij}}_x(t,x,y), b(t,x,y,p)$ are Hölder continuous in variable $x$ with exponent $\beta$ $\psi_x(t,x,y)$ is Hölder continuos in $x$ and $t$ with exponent $\beta$ and $\frac{\beta}{2}$ respectively.
     \item [(P8)] $\partial\Omega$ is of class $H^{2+\beta}$.
     \item [(P9)] $\psi(0,x,0)=0$.
   
\end{itemize}
Then there exist unique solution $u\in H^{2+\beta,1+\frac{\beta}{2}}(\Bar{\Omega}_T) $
satisfying \eqref{bmeq}--\eqref{bmeqic} and 
\begin{equation}\label{linf1}
    \|u\|_{L^{\infty}((0,T)\times \Omega)}\leq M,
\end{equation}
where $M$ is dependent on $C_0,C_1$ and $\psi_0$ and independent of $|\Omega|$.
\end{theorem}
\begin{proof}
Note that the formulation of Theorem \ref{Lad} is done using notations similar to the ones used in the monograph \cite{ladyzhenskaia1988linear}. 
The proofs of the existence and uniqueness properties follow directly from Theorem 7.4 in \cite{ladyzhenskaia1988linear}. The proof of the upper bound \eqref{linf1}  is a consequence of combining both Theorem 7.2 and Theorem 7.3 of \cite{ladyzhenskaia1988linear}.
\end{proof}
\vspace{.5cm}
\begin{lemma}[A comparison principle]\label{compar}
Assume \ref{A1}, \ref{A2} and \ref{A4} to hold true. Let $E\subset \Omega^\es$ be a bounded domain and $v\in L^2(0,T;H^1(E))\cap H^1(0,T;L^2(E))\cap L^\infty ((0,T)\times E)$, $u\in L^2(0,T;H_E)\cap H^1(0,T;L^2(E))\cap L^\infty ((0,T)\times E)$, where 
\begin{align*}
    H_E:=\{u\in H^1(E): u=0\;\; \mbox{on} \;\;\partial E\backslash (\Gamma^\es\cap \partial E)\}.
\end{align*}
 If $u,v$ satisfy the following identities
 \begin{align}
      \int_{E} \partial_t u \phi dx+\int_{E} D^\es \nabla u \nabla \phi dx - \frac{1}{\varepsilon}\int_{E} B^\es P(u) \nabla \phi dx=  \int_{E} f^\es \phi dx-\es\int_{{\Gamma_N^\es}\cap \partial E} g_N(u) \phi d\sigma,\label{l1}\\
      \int_{E} \partial_t v \psi dx+\int_{E} D^\es \nabla v \nabla \psi dx - \frac{1}{\varepsilon}\int_{E} B^\es P(v) \nabla \psi dx=  \int_{E} f^\es \psi dx-\es\int_{{\Gamma_N^\es}\cap \partial E} g_N(v) \psi d\sigma,\label{l2}
\end{align}
 and \begin{align}
 u(0)&=v(0),\\
     u&\leq v \;\;\mbox{on} \;\; \partial E\backslash (\Gamma^\es\cap \partial E).
 \end{align}
 for all $\phi\in H_E$, $\psi\in H^1(E)$,
 then it holds \begin{align}
     u\leq v \;\;\mbox{on}\;\;E.
 \end{align}
\end{lemma}
\begin{proof}
Let $(v-u)=(v-u)^{+}-(v-u)^{-}$, where
\begin{align}
    (v-u)^{-}&:=\max\{0,-(v-u)\}, \label{vmu}\\
    (v-u)^{+}&:=\max\{0,(v-u)\}. \nonumber
\end{align}
Since $u\leq v \;\;\mbox{on}\;\;\partial E\backslash(\Gamma^\es\cap \partial E)$, $(v-u)^{-}=0$ at $\partial E\backslash(\Gamma^\es\cap \partial E)$. 
In \eqref{l1} and \eqref{l2} we choose the test functions $\phi, \psi$ to be both $(v-u)^{-}$.  Subtracting the corresponding results form each other yields the expressions:
\begin{align}
    \int_{E} \partial_t (v-u) &(v-u)^- dx+\int_{E} D^\es \nabla (v-u) \nabla (v-u)^- dx\nonumber \\
    &-  \frac{1}{\varepsilon}\int_{E} B^\es (P(v)-P(u)) \nabla (v-u)^- dx= \es \int_{{\Gamma_N^\es}\cap \partial E} (g_N(u)-g_N(v)) (v-u)^- d\sigma,\label{l2e2}
\end{align}
\begin{align}
    \int_{E} \partial_t (v-u)^{-} &(v-u)^- dx+\int_{E} D^\es \nabla (v-u)^- \nabla (v-u)^- dx \nonumber\\
    &=  \frac{1}{\es}\int_{E} B^\es (P(v)-P(u)) \nabla (v-u)^- dx +\es\int_{{\Gamma_N^\es}\cap \partial E} (g_N(v)-g_N(u)) (v-u)^- d\sigma.\label{}
\end{align}

\par For $\delta>0$, we have 
\begin{align}
   \frac{1}{\es} \int_{E} B^\es (P(v)-P(u)) \nabla (v-u)^- dx&=\frac{1}{\es}\int_{E} B^\es (v-u) \nabla (v-u)^- dx\nonumber\\
   &\hspace{2cm}-\frac{1}{\es}\int_{E} B^\es  C^\es (v-u)(v+u) \nabla (v-u)^- dx\nonumber\\
    &\leq  C(\es) \int_{E}|(v-u)^-||\nabla (v-u)^-|\nonumber\\
    &{  +C(\es,||u||_{\infty},||v||_{\infty}) \int_{E}|(v-u)^-||\nabla (v-u)^-|}\nonumber\\
    &\leq C(\delta,\es)\int_{E}|(v-u)^-|^2 dx+\delta \int_{E}|\nabla(v-u)^-|^2 dx.\label{l2e3}{}
\end{align}

\noindent By \eqref{gnc3} and \eqref{vmu}, we get
\begin{align}
    \int_{{\Gamma_N^\es}\cap \partial E} (g_N(v)-g_N(u)) (v-u)^- d\sigma\leq 0.\label{l2e4}
\end{align}
Now, using \ref{A2}, \eqref{l2e3} and \eqref{l2e4} on \eqref{l2e2}, for a $\delta>0$ small enough, we obtain
\begin{align}
    \frac{1}{2}\frac{d}{dt}\|(v-u)^-\|^2+(\theta-\delta)\int_{E}|\nabla(v-u)^-|^2 dx\leq C(\delta,\es)\int_{E}|(v-u)^-|^2 dx.\label{l2e5}
\end{align}
Applying Grönwall's inequality to \eqref{l2e5}, we conclude
\begin{align*}
    (v-u)^-=0.
\end{align*}
Hence, we are led to  $u\leq v$ for a.e $x\in E$ and all $t\in [0,T]$.
\end{proof}
\vspace{.5cm}

\begin{theorem}
Assume \ref{A1}--\ref{A5} hold. Then for every fixed $\es>0$, there exists a unique weak solution  $u^\es\in L^2(0,T;H^1(\Omega^\es))\cap H^1(0,T;L^2(\Omega^\es))$ to the problem \eqref{meq}--\eqref{meqic} in the sense of Definition \ref{weakformul}
\end{theorem}
\begin{proof}
 We begin the proof with defining a boundary value problem similar to \eqref{meq}--\eqref{meqic} on a bounded domain. Then we study the existence, uniqueness and global boundedness property of the bounded domain problem. Later using some monotonicity argument we derive unique weak solution to \eqref{meq}--\eqref{meqic}.
We define $\Omega^\es_R:=(-\es R,\es R)^2\cap \Omega^\es$, $ \Gamma_{NR}^\es:=\partial \Omega^\es_R\backslash\partial (-\es R,\es R)^2  $, where $R\in \mathbb{N}$. Let $f_R^\es, g_{_R}^\es, D^\es_R, B^\es_R$ be restriction of $f^\es, g_N^\es,  D^\es, B^\es$ on $\Omega^\es_R$ respectively.
Let us consider the following boundary value problem 
\begin{align}
    \frac{\partial u_R^{\varepsilon}}{\partial t} +\mathrm{div}(-D^{\varepsilon}_R \nabla u_R^{\varepsilon}+ \frac{1}{\es}B_R^{\varepsilon}P(u_R^{\varepsilon}))&=f_R^{\varepsilon}  \,\,&\mbox{on}& \,\,&  \Omega_R^{\varepsilon} \times (0,T),\label{bmeq}\\
    (-D^{\varepsilon}\nabla u_R^{\varepsilon}+ \frac{1}{\es}B_R^{\varepsilon}P(u_R^{\varepsilon}))\cdot n_{\es}&=\es g_R^\es &\mbox{on}&   & \,\Gamma_{NR}^\es \times (0,T),\label{bmeqbcn}\\
    u_R^\es&=0 &\mbox{on}&   &\, (\partial\Omega\backslash\Gamma_{NR}^\es) \times (0,T)\label{bmeqbcd},\\
     u_R^\es(0)&=g &\mbox{in}&   &\, \overline{\Omega^\es}_R\label{bmeqic}.
\end{align}
In problem \eqref{bmeq}--\eqref{bmeqic}, we choose \begin{align}
    a_{ij}(t,x,u)=D^\es_{R_{ij}}(x),\nonumber\\
    b(t,x,u,u_{x})=B\cdot \nabla (u^\es_R(1-(u^\es_R)))-f^\es_R,\label{driftb}\\
    \psi (t,x,u)= g_N(u^\es),\nonumber\\
    \psi_0(x)=g(x)\nonumber.
\end{align}
Using assumption \ref{A2}, we find that  \eqref{bmeqbcn} is equivalent to \eqref{Lp2}. Using assumption \ref{A1} we obtain $D^\es_{R_{ij}}(x)$ satisfies condition $(P1),(P2),(P6),(P7)$. We use assumption \ref{A2} and we get
\begin{align}
    \nabla B (u^\es_R(1-u^\es_R))= B \nabla u^\es_R-B2u^\es_R\nabla (u^\es_R).\label{exp1}
\end{align}
By Young's inequality and by \eqref{exp1} we verify that the the term \eqref{driftb} satisfies $(P3),(P4),(P6)$ and $(P7)$.
From \ref{A4} and using \ref{A5} we have that $g_N$ satisfies (P2), (P5) and (P9).

Now, by Theorem \ref{Lad} we deduce that there exists a unique $u^\es_R\in H^{2+\beta,1+\frac{\beta}{2}}(\Bar{\Omega})$ which solves \eqref{bmeq}--\eqref{bmeqic} and 
\begin{align}
    \|u^\es_R\|_{L^{\infty}((0,T)\times \Omega_R^\es)}&\leq M,\label{blinf}
    \end{align}
    where $M$ is a constant independent of $R$.\\

\noindent Using arguments similar to those involved in the proof of Theorem \ref{energy} and jointly with \ref{A3}, we obtain 
\begin{align}
\max_{t\in (0,T)}\|u_R^\es \|_{L^2(\Omega_R^\es)}\leq C,\label{bl2}\\
\|u_R^\es \|_{L^2(0,T;H^1(\Omega_R^\es))}\leq C\label{bgrad},
\end{align}
where $C$ is independent of $R$. Following similar steps of Theorem \ref{posit}, we obtain
\begin{align}
    0\leq u_R^\es,\label{pos}
\end{align}
for a.e. $x\in \Omega^\es_R$ and all $t\in (0,T).$

\noindent Now, we define an extension of the bounded domain problem in unbounded domain.\\ { In particular we set}
\begin{align}
    \Tilde{u}_R^\es=\begin{cases} 
      u_R^\es & x\in \Omega_R^\es \\
     0 & x\in \Omega^\es\backslash \Omega_R^\es .
   \end{cases}\label{bext}
\end{align}
\par From \eqref{blinf}--\eqref{bgrad} and \eqref{bext}, we get
\begin{align}
    \|\Tilde{u}^\es_R\|_{L^{\infty}((0,T)\times \Omega_R^\es)}\leq M,\label{2blinf}\\
    \max_{t\in (0,T)}\|\Tilde{u}_R^\es \|_{L^2(\Omega_R^\es)}\leq C,\label{2bl2}\\
\|\Tilde{u}_R^\es \|_{L^2(0,T;H^1(\Omega_R^\es))}\leq C.\label{2bgrad}
\end{align}
Using \eqref{2blinf}--\eqref{2bgrad},  Lemma \ref{compar} and \eqref{pos} together with Monotone Convergence Theorem, Banach-Alaoglu theorem (see \cite{rudin1973functional}) there exists $u^\es \in L^2(0,T;H^1(\Omega^\es))\cap H^1(0,T;L^2(\Omega^\es))$ such that
\begin{align}
    \Tilde{u}_R^\es \rightarrow u^\es\hspace{2cm}&\mbox{on}\hspace{1cm}L^{2}((0,T)\times\Omega^\es),\label{con1}\\
   \nabla \Tilde{u}_R^\es \rightharpoonup \nabla u^\es \hspace{2cm}&\mbox{on}	\hspace{1cm}L^{2}(0,T;L^{2}(\Omega^\es)),\label{con2}\\
    \frac{\partial (\Tilde{u}_R^\es)}{\partial t} \rightharpoonup  \frac{\partial (u^\es)}{\partial t} \hspace{2cm}&\mbox{on}	\hspace{1cm}L^{2}(0,T;L^{2}(\Omega^\es)),\label{con3}\\
     P(\Tilde{u}_R^\es) \rightharpoonup P( u^\es) \hspace{2cm}&\mbox{on}	\hspace{1cm}L^{2}(0,T;L^{2}(\Omega^\es)),\label{con5}\\
      g_N(\Tilde{u}_R^\es) \rightharpoonup g_N( u^\es) \hspace{2cm}&\mbox{on}	\hspace{1cm}L^{2}(0,T;L^{2}(\Gamma_N^\es))\label{con6}.
    \end{align}

\noindent Convergences \eqref{con1}--\eqref{con6} guarantee the existence of  weak solution to the problem \eqref{meq}--\eqref{meqic} in the sense of Definition \ref{weakformul}.

\noindent To prove the uniqueness of the weak solution , we consider $u^\es_1,u^\es_2\in L^2(0,T;H^1(\Omega^\es))$ weak solutions to problem \eqref{meq}-- \eqref{meqic} in the sense of Definition \eqref{weakformul}. Then we have 
\begin{align}
    \int_{\Omega^\es} {(u^\es_1)}_t \phi dx + \int_{\Omega^\es} D^\es\nabla u^\es_1 \nabla \phi dx =  \frac{1}{\es}\int_{\Omega^\es} B^\es P(u^\es_1)\nabla \phi dx+\int_{\Omega^\es}f^\es \phi dx- \int_{\Gamma_N^\es} g_N(u^\es_1) \phi d\sigma,\label{u1}\\
   \int_{\Omega^\es} {(u^\es_2)}_t \phi dx + \int_{\Omega^\es} D^\es\nabla u^\es_2 \nabla \phi dx =  \frac{1}{\es}\int_{\Omega^\es} B^\es P(u^\es_2)\nabla \phi dx+\int_{\Omega^\es}f^\es \phi dx- \int_{\Gamma_N^\es} g_N(u^\es_2) \phi d\sigma,\label{u2}
 \end{align}
for all $\phi \in L^2(0,T;H)$.\\

We choose the test function $\phi:= u^\es_1-u^\es_2$ for \eqref{u1} and \eqref{u2} and substracting \eqref{u2} to \eqref{u1} we obtain
\begin{align}
\int_{\Omega^\es} \phi_t \phi dx + \int_{\Omega^\es}D^\es \nabla \phi\nabla \phi dx =  \frac{1}{\es}\int_{\Omega^\es} B^\es (P(u_1^\es)-P(u_2^\es))\nabla \phi dx+\int_{\Gamma_N^\es}( g_N(u_2^\es)-g_N(u_1^\es) )\phi d\sigma.
\end{align}
From \ref{A4}, we get 
\begin{align}
    \int_{\Gamma_N^\es}( g_N(u^\es_2)-g_N(u^\es_1) )\phi d\sigma\leq0\label{gn0}.
\end{align}
Using the Mean Value Theorem, \eqref{2blinf}, and the structure of $P(\cdot)$ we get
\begin{align}
   \frac{1}{\es} \int_{\Omega^\es} B^\es (P(u^\es_1)-P(u^\es_2))\nabla \phi dx\leq C(\es)\label{pu1} \int_{\Omega^\es} \phi\nabla \phi dx,
\end{align}
where $C$ is constant determined by 
\begin{align}
    C=\sup_{r\in [0,M]} P'(r),
\end{align}
while $M$ is the constant from \eqref{2blinf}.
Using the ellipticity condition \eqref{A1} together with \eqref{gn0} and \eqref{pu1}, we obtain
\begin{align}
    \frac{d}{dt}\|\phi\|^2_{L^2(_{\Omega^\es})}+\theta \int_{\Omega^\es} |\nabla \phi|^2 dx \leq C\int_{\Omega^\es} \phi\nabla \phi dx\nonumber.
\end{align}
Using Young's inequality and choosing $\delta$ small enough, we get 
\begin{align}
    \frac{d}{dt}\|\phi\|^2_{L^2(_{\Omega^\es})}+(\theta-\delta) \int_{\Omega^\es} |\nabla \phi|^2 dx \leq C(\delta)\int_E \phi^2 dx,\nonumber\\
    \frac{d}{dt}\|\phi\|^2_{L^2(_{\Omega^\es})} \leq C(\delta)\int_{\Omega^\es} \phi^2 dx.\nonumber
\end{align}
Now, using Grönwall's inequalty, we get $\phi=0$ a.e. in $\Omega^\es$ for all times $t\in (0,T)$. Hence the uniqueness of the weak solution to problem \eqref{meq}--\eqref{meqic} follows.
\end{proof}

Note that in Theorem \ref{Lad} we are considering  Neumann boundary conditions. In  our problem \eqref{bmeq}--\eqref{bmeqic}, we are dealing with a Dirichlet boundary condition imposed across a part of the boundary. Since it is a homogeneous condition and also the sets $(\partial\Omega\backslash\Gamma_{NR}^\es)$ and $\Gamma_{NR}^\es$ are disjoint, we can adapt the proof of Theorem \ref{Lad} to our setup \eqref{bmeq}--\eqref{bmeqic}.

\vspace{.5cm}
\subsection{Energy estimates}
\begin{theorem}\label{energy}
Assume \ref{A1}--\ref{A5} hold. There exists a constant $C>0$ independent of $\varepsilon$ such that the following energy estimate hold
\begin{align}
    \|u^\varepsilon\|_{L^{2}(0,T;{H^1(\Omega^\es)})}\leq C,\label{energy1}\\
    \|u^\varepsilon\|_{L^{\infty}(0,T;L^{2}(\Omega^\es)}\leq C,\label{energy2}\\
      \|u^\varepsilon\|_{L^{\infty}((0,T)\times \Omega^\es)}\leq C,\label{energy4}\\
       \|u^\varepsilon\|_{L^{2}(0,T;{L^2(\Gamma_N^\es)})}\leq C. \label{energy5}
\end{align}
\end{theorem}
\begin{proof}
We prove the estimate \eqref{energy1} by choosing the test function $\phi= u^\es$ in the weak formulation \eqref{wf}, we get
\begin{align}
   \frac{1}{2}\frac{d}{dt}\|u^\es \|_{L^2(\Omega^\es)}+\theta \int_{\Omega^\es}|\nabla u^\es|^2 dx\leq \frac{1}{\es}\int_{\Omega^\es}B^\es P(u^\es)\nabla u^\es dx + \int_{\Omega^\es}f^\es u^\es dx-\es\int_{{\Gamma_N^\es}} g_N(u^\es) u^\es d\sigma.\label{ee1}
   \end{align}
   We define 
   \begin{align}
       \Tilde{P}(u^\es):=\frac{(u^\es)^2}{2}-C^\es\frac{(u^\es)^3}{3},\label{ptilda}
   \end{align}
   then we have 
   \begin{align}
       \nabla \Tilde{P}(u^\es) &=P(u^\es) \nabla u^\es .\label{pprime}
   \end{align}
   Now, using \ref{A2}, \eqref{pprime}, we have
   \begin{align}
    \int_{\Omega^\es}B^\es P(u^\es)\nabla u^\es dx &= \int_{\Omega^\es}B^\es \nabla \Tilde{P}(u^\es) dx\nonumber \\
    &={-\int_{\Omega^\es}\nabla\cdot B^\es \tilde{P}(u^\es) dx} + \int_{\Gamma_N^\es}B^\es\cdot n \Tilde{P}(u^\es) dx \nonumber\\
    &=0\label{zer}.
\end{align}

\noindent Using \eqref{zer} and Grönwall's inequality from \eqref{ee1} we get the required results, i.e. \eqref{energy1} and \eqref{energy2}.\\

\par The estimate \eqref{energy4} follows directly from \eqref{2blinf}.
Since $\Gamma_N^\es$ is uniformly Lipschitz, we use the trace result stated in Theorem 15.23 of \cite{leoni2017first} and we obtain \eqref{energy5}.
\vspace{.5cm}

\end{proof}
\begin{theorem}[A positivity result]\label{posit}
Assume \ref{A1}--\ref{A5} hold. Let $u^\es$ be a weak solution of \eqref{meq}--\eqref{meqic}. Then $u^\es\geq 0$.
\end{theorem}

\begin{proof}
We choose as test function $\phi=(u^{\es})^-$ in \eqref{wf}, where $u^\es=(u^{\es})^+-(u^{\es})^-$ with $u^-=\max \{0,-u^\es\}$ and $u^+=\max \{0,u^\es\}$ and we get 
\begin{align}
      \int_{\Omega^\es} \partial_t u^\es (u^{\es})^- dx+\int_{\Omega^\es} D^\es \nabla u^\es \nabla (u^{\es})^- dx - \frac{1}{\es}\int_{\Omega^\es} B^\es P(u^\es) \nabla (u^{\es})^- dx\nonumber\\
      =  \int_{\Omega^\es} f^\es (u^{\es})^- dx-\int_{{\Gamma_N^\es}} g_N(u^\es) (u^{\es})^- d\sigma.\label{pos1}
\end{align}

Using   (A1), we have
{
\begin{align}
      \frac{1}{2}\frac{d}{dt}||(u^\es)^-||^2_{L^2(\Omega^\es)}+\theta\int_{\Omega^\es} &| \nabla (u^\es)^- |^2 dx + \frac{1}{\es}\int_{\Omega^\es} B^\es P((u^\es)) \nabla (u^{\es})^- dx\nonumber\\
      &\leq -\int_{\Omega^\es} f^\es (u^{\es})^- dx+\es\int_{{\Gamma_N^\es}} g_N((u^\es)^-) (u^{\es})^- d\sigma. \label{pos2}
\end{align}
}

After an integration by parts with respect to the space variable and \ref{A2} , we get
{ \begin{align}
   { \frac{1}{\es} }\int_{\Omega^\es}B^{\es} P(-(u^\es)^-)\nabla (u^\es)^- dx &=\frac{-1}{\es}\int_{\Omega^\es}B^{\es}\left( (u^\es)^-(1+C^\es(u^\es)^-)\nabla (u^\es)^-\right) dx\nonumber\\
   &=\frac{-1}{\es}\int_{\Omega^\es}B^{\es} \nabla\frac{1}{2}((u^\es)^-)^2-B^{\es}C^\es\frac{1}{3}\nabla((u^\es)^-)^3 dx\nonumber\\
   &=\frac{1}{\es}\int_{\Omega^\es}\nabla B^{\es} \frac{1}{2}((u^\es)^-)^2+\nabla (B^{\es}C^\es)\frac{1}{3}((u^\es)^-)^3 dx\nonumber\\
   &\hspace{1cm}+\frac{1}{\es} \int_{\Gamma_N^\es}B^{\es}\cdot n \left(\frac{1}{2}((u^\es)^-)^2+C^\es\frac{1}{3}((u^\es)^-)^3\right) d\sigma\nonumber\\
    &=0\label{pos3}.
\end{align}}
Since 
{\begin{equation}\nonumber
    \int_{{\Gamma_N^\es}} g_N(u^\es) (-u^{\es})^- d\sigma=\int_{{\Gamma_N^\es}} g_N((-u^{\es})^-) (-u^{\es})^- d\sigma,
\end{equation}
by \eqref{gnc2} and \ref{A3} we get
\begin{align}
   -\int_{\Omega^\es} f^\es (u^{\es})^- dx+ \int_{{\Gamma_N^\es}} g_N((-u^{\es})^-) (-u^{\es})^- d\sigma\leq 0.\label{pos4}
\end{align}}

Combining \eqref{pos2}, \eqref{pos3} and \eqref{pos4}, we obtain
\begin{align}
     \frac{1}{2}\frac{d}{dt}||(u^\es)^-||^2_{L^2(\Omega^\es)}\leq  0\label{pos6}
\end{align}
As a direct application of Grönwall's inequality (see Appendix B of \cite{evans2010partial})  on \eqref{pos6}, we conclude 
\begin{align}
    \|(u^\es)^-\|^2_{L^2(\Omega^\es)}\leq 0.\nonumber
\end{align}
Hence $u^\es\geq 0$ a.e. in $\Omega^\es$ and for all $t\in (0,T)$.
\end{proof}
\vspace{.5cm}
\section{Derivation of the upscaled model}\label{upscalemodel}
In this section we pass $\es\rightarrow 0$ to the homogenization limit and derive the corresponding upscaled equation and effective coefficients.
\subsection{Extension to fixed domain}
\begin{lemma}\label{extension}
If $u^\es\in H^1(\Omega^\es)\cap L^\infty(\Omega^\es)$, where $\Omega^\es$ defined as in \eqref{ome}. Then there exist a constant $C>0$ and an extension of $u^\es$ to $H^1(\mathbb{R}^2)$ denoted by  $\Tilde{u}^\es$ such that
\begin{align}
{\Tilde{u}^\es}{|_{\Omega^\es}}&={u}^\es,\label{ext1}\\
 \| \Tilde{u}^\es\|_{L^2(\mathbb{R}^2)}&\leq C \| {u}^\es\|_{L^2(\Omega^\es)},\label{ext2}\\
    \|\nabla \Tilde{u}^\es\|_{L^2(\mathbb{R}^2)}&\leq C \|\nabla {u}^\es\|_{L^2(\Omega^\es)},\label{ext3}\\
    \| \Tilde{u}^\es\|_{L^\infty(\mathbb{R}^2)}&\leq C \| {u}^\es\|_{L^\infty(\Omega^\es)}.
    \end{align}
\end{lemma}

\begin{proof}
Define $\Omega_R^\es :=\left((-\es R,+\es R)\times (-\es R,+\es R)\right)\cap \Omega^\es $, where $R\in \mathbb{N}$, clearly $\Omega_R^\es \cap \Omega_0^\es =\emptyset$. 
Let $u_R^\es:=u_{|_{\Omega_R^\es}}^\es$. Since $u^\es\in H^1(\Omega^\es)$ we have $u_R^\es\in H^1({\Omega_R^\es})$. By using Theorem 2.1 of \cite{ACERBI1992481}  there exist a $\Tilde{u}\in H^1((-\es R,+\es R)\times (-\es R,+\es R))$ and $C>0$ which is independent of $R$ such that 
\begin{align}
{\Tilde{u}_R^\es}{|_{\Omega_R^\es}}&={u}_R^\es,\label{exe0}\\
 \|\Tilde{u}_R^\es\|_{L^2((-\es R,+\es R)\times (-\es R,+\es R))} &\leq C \|u_R^\es\|_{L^2(\Omega_R^\es)},\label{exe1}\\
 \|\nabla\Tilde{u}_R^\es\|_{L^2((-\es R,+\es R)\times (-\es R,+\es R))} &\leq C \|\nabla u_R^\es\|_{L^2(\Omega_R^\es)},\label{exe2}\\
 \|\Tilde{u}_R^\es\|_{L^\infty((-\es R,+\es R)\times (-\es R,+\es R))}& \leq C \|u_R^\es\|_{L^\infty(\Omega^\es)},\nonumber\\
 & \leq C \|u^\es\|_{L^\infty(\Omega^\es)}\label{exe3}.
\end{align}
The inequalities \eqref{exe1} and \eqref{exe2} follows directly from Theorem 2.1 of \cite{ACERBI1992481}.

The inequality \eqref{exe3} follows from the fact that the proof of the extension  operator contains a standard reflection argument, hence it preserves the $L^\infty$ bound in the extended regions(see Theorem 9.7 of \cite{brezis2011functional}). 
\par Now,  we prove the following identity
\begin{align}
    {\Tilde{u}_{R+N}^\es}{|_{\Omega_R^\es}}={\Tilde{u}_R^\es}\;\;\; \mbox{on} \;\;\;\Omega_R^\es \label{exe4}
\end{align}
for all $N\in \mathbb{N}$. 
From the definition of $u_{R}^\es$ and $u_{R+1}^\es$, we have 
\begin{align}
    u_{R+1}^\es|_{\Omega_R^\es}=u_{R}^\es.
\end{align}
We define the extension of $u_{R+1}^\es$ in such a way that, if $\Tilde{u}_R^\es$ is the extension of $ u_R^\es$, then
\begin{align}
    {\Tilde{u}_{R+1}^\es}{|_{\Omega_R^\es}}=\Tilde{u}_R^\es\nonumber
\end{align}
and in $(-\es (R+1),\es (R+1))^2\backslash (-\es R,\es R)^2$, we extend $u_{R+1}$ using a reflection argument similar to Theorem 9.7 of \cite{brezis2011functional} and \cite{ACERBI1992481}. Now inductively we get \eqref{exe4}.

\par We define extension of $u^\es$ on $\mathbb{R}^2$ as 
$\Tilde{u}^\es$ where $\Tilde{u}^\es$ is defined as for any $x\in \mathbb{R}^2$ there exists $R\in \mathbb{N}$ such that $x\in (-\es R,+\es R)\times (-\es R,+\es R))$ and 
\begin{align}
    \Tilde{u}^\es(x)= {\Tilde{u}_R^\es}(x).\label{linf}
\end{align}
By identity \eqref{exe4}, $\Tilde{u}^\es(x)$ is well defined function on $\mathbb{R}^2$.
Using Fatou's Lemma, \eqref{exe1} and \eqref{ext2},  we have
\begin{align}
    \int_{\mathbb{R}^2}(\Tilde{u}^\es)^2dx &=\int_{\mathbb{R}^2}\lim_{R\rightarrow\infty}\chi_{((-\es R,+\es R)\times (-\es R,+\es R))} (\Tilde{u}^\es)^2dx\nonumber\\
    &=\int_{\mathbb{R}^2}\lim_{R\rightarrow\infty}\chi_{((-\es R,+\es R)\times (-\es R,+\es R))} (\Tilde{u}_R^\es)^2dx\nonumber\\
   & \leq \liminf_{R\rightarrow\infty} \int_{\mathbb{R}^2}\chi_{((-\es R,+\es R)\times (-\es R,+\es R))} (\Tilde{u}_R^\es)^2dx\nonumber\\
  & \leq C \|u_R^\es\|_{L^2(\Omega_R^\es)}^2\nonumber\\
  &\leq C \|u^\es\|_{L^2(\Omega^\es)}^2\nonumber.
\end{align}
Similarly, we prove 
\begin{align}
    \int_{\mathbb{R}^2}|\nabla \Tilde{u}^\es|^2dx &=\int_{\mathbb{R}^2}\lim_{R\rightarrow\infty}\chi_{((-\es R,+\es R)\times (-\es R,+\es R))} |\nabla \Tilde{u}^\es|^2dx\nonumber\\
    &=\int_{\mathbb{R}^2}\lim_{R\rightarrow\infty}\chi_{((-\es R,+\es R)\times (-\es R,+\es R))} |\nabla \Tilde{u}_R^\es|^2dx\nonumber\\
   & \leq \liminf_{R\rightarrow\infty} \int_{\mathbb{R}^2}\chi_{((-\es R,+\es R)\times (-\es R,+\es R))} |\nabla \Tilde{u}_R^\es|^2 dx\nonumber\\
  & \leq C \|\nabla {u}_R^\es\|_{L^2(\Omega_R^\es)}^2\nonumber\\
  &\leq C \|\nabla u^\es\|_{L^2(\Omega^\es)}^2\nonumber.
\end{align}
The inequality
   \begin{equation}\nonumber
      \| \Tilde{u}^\es\|_{L^\infty(\mathbb{R}^2)}\leq C \| {u}^\es\|_{L^\infty(\Omega^\es)} 
   \end{equation}
    follows from \eqref{exe3} and \eqref{linf}.
\end{proof}

It should be noted that our proof of the extension Lemma \ref{extension} relies on the fact that at the level of the standard
cell $Z$, the obstacle does not touch the external boundary of $Z$.

\subsection{Strong convergence}

In this section we prove that the sequence of solutions to the microscopic problem \eqref{wf}  strongly converges to a limit formulated in moving co-ordinates. To obtain the wanted result, we adapt to our setup  similar techniques as discussed also in \cite{hutridurga},\cite{marusik2005} and \cite{ijioma2019fast}.

\begin{lemma}\label{lemmaF}
Let $f\in L^2(Z), g\in L^2(\Gamma_N)$ be given functions. Then the boundary value problem
\begin{align}
     -\Delta_y v=f \hspace{.4cm}&\mbox{on} \hspace{.4cm} &  Z,\nonumber\\
   \nabla v\cdot n=g \hspace{.4cm}&\mbox{on} \hspace{.4cm} &  \Gamma_N,\nonumber\\
    v\,\, &\mbox{is Z-periodic},&&\nonumber
\end{align}
 has a unique solution  $v\in H^1_{\#}(Z)/\mathbb{R}$ if and only if the compatibility condition
\begin{equation}\nonumber
    \int_Z f\, dy=\int_{\partial Z} g\,d\sigma_y
\end{equation}
is satisfied.
\end{lemma}
\begin{proof}
The proof follows via standard argument involving Fredholm alternative.
\end{proof}

\vspace{0.4cm}

\noindent We define
\begin{equation}\label{bstar}
    B^* e_i:=\dfrac{\int_YB(y)e_idy}{|Z|},
\end{equation}
\begin{align}
    \Omega^{\es}(t):=\left\{x+\frac{B^*t}{\es}:x\in \Omega^\es\right\},\label{omees}
\end{align}
\begin{equation}\label{ves}
    v^\es(t,x):=u^\es \left(t,x+\dfrac{B^*t}{\es}\right).
\end{equation}
We define $\{e_j\}, j\in \mathbb{Z}^2$ as orthonormal
basis for $L^2((0,1)^2)$ with compact support and $C^2$ differentiable. Related to $e_j$, we define 
\begin{align}
    e_{j,k}(x):= e_j(x-k),\\
    q_{j,k}(t,x):= e_{j,k}\left(x-\dfrac{B^*t}{\es}\right).
\end{align}
Note that $e_{j,k}$ and $q_{j,k}$ forms orthonormal basis for $L^2(\mathbb{R}^2)$. We have 
\begin{align}
    v^\es(t,x)=\sum_{j\in \mathbb{N},k\in \mathbb{Z}^2}v^\es_{jk}(t)e_{jk}(x),\\
    \Tilde{v}^\es(t,x)=\sum_{j\in \mathbb{N},k\in \mathbb{Z}^2}\Tilde{v}^\es_{jk}(t)e_{jk}(x),
\end{align}
where  $\Tilde{v}^\es(t,x)$ is the extension of $v^\es(t,x)$ on $\mathbb{R}^2$ as defined in Lemma \ref{extension} and 
\begin{align}
    v^\es_{jk}(t):=\int_{\Omega^{\es}(t)}v^\es(t,x)e_{jk}(x) dx,\label{vjk}\\
    \Tilde{v}^\es_{jk}(t):=\int_{\mathbb{R}^2}\Tilde{v}^\es(t,x)e_{jk}(x) dx.\label{vjktild}
\end{align}

\begin{lemma}\label{lemma4}
Assume \ref{A1}--\ref{A5} hold. Then for all $\delta>0$, there exist a $R_\delta>0$ such that 
\begin{align}
    \left\|u^\es\left(t,x+\dfrac{B^*t}{\es}\right)\right\|_{L^2(\Omega^{\es}(t)\backslash\Omega(R_\delta,\es))}\leq \delta\label{comp}
\end{align}
where $\Omega(R_\delta,\es):=\left(\Omega^\es\cap (-R_\delta,R_\delta)^2\right).$
\end{lemma}
\begin{proof}

We choose a specific cutoff function $\psi\in C^\infty(\mathbb{R}^2)$  such that 
\begin{align}
 \psi(x)=  \begin{cases}
      0 \;\;\mbox{for}\;\; x\in [-1,1]\\
      1\;\;\mbox{for}\;\; x\in [-2,2]^c
\end{cases} \nonumber
\end{align}
and $0\leq \psi(x)\leq 1$ for $x\in (-2,-1)\cup (1,2)$ . Now, for $x\in \mathbb{R}^2$, we define 
\begin{align}\label{psir}
    \psi_R(x):=\psi\left(\dfrac{|x|}{R}\right),\\
    \psi_R^\es(t,x):=\psi_R\left(x-\dfrac{B^*t}{\es}\right).
\end{align}

Consider the weak formulation \eqref{wf}. Integrating it over $(0,t)$ and choosing $\phi(t,x)=\psi_R^\es(t,x)u^\es(t,x)$, we get
\begin{multline}\label{testpsiu}
     \int_0^t \int_{\Omega^\es} \partial_t u^\es (\psi_R^\es(t,x)u^\es) dxdt+\int_0^t\int_{\Omega^\es} D^\es \nabla u^\es \nabla (\psi_R^\es(t,x)u^\es) dxdt
     \\ -\dfrac{1}{\es} \int_0^t\int_{\Omega^\es} B^\es u^\es(1-C^\es u^\es) \nabla (\psi_R^\es(t,x)u^\es) dxdt
      \\=  \int_0^t\int_{\Omega^\es} f^\es (\psi_R^\es(t,x)u^\es) dxdt
      -\es \int_0^t\int_{{\Gamma_N^\es}} g_N(u^\es) (\psi_R^\es(t,x)u^\es) d\sigma dt.
\end{multline}
Using integration by parts with respect to time variable, we have
\begin{multline}\nonumber
    \int_0^t\int_{\Omega^\es} \partial_t u^\es (\psi_R^\es(t,x)u^\es) dxdt=\dfrac{1}{2}\dfrac{B^*}{\es}\int_0^t\int_{\Omega^\es}  (u^\es)^2 \nabla \psi_R^\es(t,x) dxdt+\dfrac{1}{2} \int_{\Omega^\es} (u^\es(t,x))^2 \psi_R^\es(x,t) dx\\
    -\dfrac{1}{2} \int_{\Omega^\es} (u^\es(0,x))^2 \psi_R^\es(0,x) dx,
\end{multline}
while the second term in the left hand side of \eqref{testpsiu} becomes
\begin{align}
    \int_0^t\int_{\Omega^\es} D^\es \nabla u^\es \nabla (\psi_R^\es(t,x)u^\es) dx=\int_0^t\int_{\Omega^\es} D^\es \nabla u^\es \nabla \psi_R^\es(t,x) u^\es dxdt+\int_0^t\int_{\Omega^\es} D^\es \nabla u^\es \nabla u^\es  \psi_R^\es(x,t) dxdt\nonumber.
\end{align}
We have
\begin{multline}\nonumber
    \dfrac{1}{\es} \int_0^t\int_{\Omega^\es} B^\es u^\es(1-C^\es u^\es) \nabla (\psi_R^\es(t,x)u^\es) dxdt=\dfrac{1}{\es} \int_0^t\int_{\Omega^\es} B^\es u^\es \psi_R^\es(t,x)\nabla u^\es dxdt\\
    +\dfrac{1}{\es} \int_0^t\int_{\Omega^\es} B^\es u^\es u^\es \nabla  \psi_R^\es(t,x) dxdt{ -}\dfrac{1}{\es} \int_0^t\int_{\Omega^\es} B^\es C^\es \nabla u^\es \psi_R^\es(t,x)(u^\es)^2 dxdt\\
    { -}\dfrac{1}{\es}\int_0^t\int_{\Omega^\es} B^\es C^\es (u^\es)^3  \nabla \psi_R^\es(t,x) dxdt,
\end{multline}
integration by parts with respect to the space variable and \ref{A2}, we get
\begin{multline}\nonumber
    \dfrac{1}{\es} \int_0^t\int_{\Omega^\es} B^\es u^\es(1-C^\es u^\es) \nabla (\psi_R^\es(t,x)u^\es) dxdt=-\dfrac{1}{\es} \int_0^t\int_{\Omega^\es} B^\es u^\es \nabla (\psi_R^\es(t,x) u^\es) dxdt\\
    +\dfrac{1}{\es} \int_0^t\int_{\Omega^\es} B^\es u^\es u^\es \nabla  \psi_R^\es(t,x) dxdt+\dfrac{1}{\es} \int_0^t\int_{\Omega^\es} B^\es C^\es  u^\es \nabla(\psi_R^\es(t,x)(u^\es)^2) dxdt\\
    -\dfrac{1}{\es}\int_0^t\int_{\Omega^\es} B^\es C^\es (u^\es)^3  \nabla \psi_R^\es(t,x) dxdt,
\end{multline}
\begin{multline}\nonumber
    \dfrac{1}{\es} \int_0^t\int_{\Omega^\es} B^\es u^\es(1-C^\es u^\es) \nabla (\psi_R^\es(t,x)u^\es) dxdt=
    \dfrac{1}{2\es} \int_0^t\int_{\Omega^\es} B^\es u^\es u^\es \nabla  \psi_R^\es(t,x) dxdt\\
    -\dfrac{1}{2\es}\int_0^t\int_{\Omega^\es} B^\es C^\es (u^\es)^3  \nabla \psi_R^\es(t,x) dxdt,
\end{multline}
\begin{multline}
  \dfrac{1}{2\es}B^*\int_0^t\int_{\Omega^\es}  (u^\es)^2 \nabla \psi_R^\es(t,x) dxdt+\dfrac{1}{2} \int_{\Omega^\es} (u^\es(t,x))^2 \psi_R^\es(t,x)) dx
  -\dfrac{1}{2} \int_{\Omega^\es} (u^\es(0,x))^2 \psi_R^\es(0,x)) dx\\
  +\int_0^t\int_{\Omega^\es} D^\es \nabla u^\es \nabla \psi_R^\es(t,x) u^\es dxdt+\int_0^t\int_{\Omega^\es} D^\es \nabla u^\es \nabla u^\es  \psi_R^\es(t,x) dxdt
  \\- \dfrac{1}{2\es}\int_0^t\int_{\Omega^\es} B^\es u^\es u^\es \nabla  \psi_R^\es(t,x) dxdt
    +\dfrac{1}{2\es}\int_0^t\int_{\Omega^\es} B^\es C^\es (u^\es)^3  \nabla \psi_R^\es(t,x) dxdt
    \\= \int_0^t\int_{\Omega^\es} f^\es (\psi_R^\es(t,x)u^\es) dxdt
      -\es \int_0^t\int_{{\Gamma_N^\es}} g_N(u^\es) (\psi_R^\es(t,x)u^\es) d\sigma dt.\label{st2}
\end{multline}
Using \ref{A5}  and the definition of $\psi_R^\es$, we have
\begin{align}
    \lim_{R\rightarrow \infty}\dfrac{1}{2} \int_{\Omega^\es} (u^\es(0,x))^2 \psi_R^\es(0,x)) dx=0.\label{stini}
\end{align}
By the definition \eqref{psir}, we have
\begin{align}\label{psibound}
    |\nabla \psi_R^\es|, |\nabla\cdot\nabla \psi_R^\es|\leq \frac{C}{R},
\end{align}
where $C>0$ is a constant.
Using \ref{A1}, \eqref{psibound} and Theorem \ref{energy}, we conclude
\begin{align}
\lim_{R\rightarrow\infty}\int_0^t\int_{\Omega^\es} D^\es \nabla u^\es \nabla \psi_R^\es(t,x) u^\es dxdt=0.\label{st3}
\end{align}
Consider the following auxiliary problem: Find $G_i$ (with $i\in \{1,2\}$) such that
\begin{align}
    -\Delta_y G_i&=B^*e_i-B(y)e_i\;\;\mbox{on}\;\;Z,\label{aux1}\\
    \nabla G_i\cdot n&= 0  \;\;\mbox{on}\;\;\Gamma_N,\\
    G_i &\mbox{ is $Z$- periodic}.\label{aux1bc}
\end{align}
Using Lemma \ref{lemmaF}, definition \eqref{bstar}, there exist a weak solution $G_i\in  H^1_{\#}(Z)/\mathbb{R}$ satisfying \eqref{aux1}--\eqref{aux1bc}.
Now using the change of variable $y= \frac{x}{\es}$ and defining $G_i^\es(x):=G_i(\frac{x}{\es})$,  we get
\begin{align}
    -\es^2\Delta_x G_i^\es&=B^*e_i-B^\es(x)e_i\;\;\mbox{on}\;\;\Omega_\es,\label{aux2}\\
   \es \nabla G_i^\es\cdot n&= 0 \;\;\mbox{on}\;\;\Gamma_N^\es.\label{aux2bc}
\end{align}
Thanks to the auxiliary problem \eqref{aux2}--\eqref{aux2bc} we can write
\begin{align}
    \dfrac{1}{2\es}\int_0^t\int_{\Omega^\es}(B^*-B)  (u^\es)^2 \nabla \psi_R^\es(t,x) dxdt&=\sum_{i=1}^{2}\frac{1}{2}\int_0^t\int_{\Omega^\es} \es\nabla G_i^\es \nabla( (u^\es)^2 \dfrac{\partial \psi_R^\es(t,x)}{\partial x_i}) dxdt\label{st1}.
\end{align}
We observe that 
$$\es\nabla G_i^\es=\nabla_y G(\frac{x}{\es}),$$
and 
$$|\nabla( (u^\es)^2 \dfrac{\partial \psi_R^\es(t,x)}{\partial x_i})|\leq |u^\es\nabla u^\es||\nabla\psi_R|+|(u^\es)^2\nabla\cdot\nabla \psi_R|.$$
By  Theorem \ref{energy} and by \eqref{psibound} we can now estimate the right hand side of \eqref{st1}
in the following way
\begin{align}
    \dfrac{1}{2\es}\int_0^t\int_{\Omega^\es}(B^*-B)  (u^\es)^2 \nabla \psi_R^\es(t,x) dxdt&\leq \dfrac{C}{R}.\label{unbnd} 
\end{align}
Using Theorem \ref{posit}, \ref{A4} and  observing that $\psi_R^\es$ is nonnegative, we have
\begin{equation}
     -\es \int_0^t\int_{{\Gamma_N^\es}} g_N(u^\es) (\psi_R^\es(t,x)u^\es) d\sigma dt\leq 0.\label{stbdr}
\end{equation}
To obtain an estimate on the nonlinear term, we consider the following auxiliary problem: Find $H_i$ (with $i\in \{1,2\}$) such that
\begin{align}
    -\Delta_y H_i&=BCe_i\;\;\mbox{on}\;\;Z,\label{aux3}\\
    \nabla H_i\cdot n&= 0 \;\;\mbox{on}\;\;\Gamma_N,\\
    H_i &\mbox{ is $Z$- periodic}\label{aux3bc}.
\end{align}
Using Lemma \ref{lemmaF}, \eqref{bc}, there exists a weak solution $H_i\in  H^1_{\#}(Z)/\mathbb{R}$ satisfying \eqref{aux3}--\eqref{aux3bc}.
Using the change of variable $y=\frac{x}{\es}$ and defining $H_i^\es:=H_i(\frac{x}{\es})$, we have 
\begin{align}
    -\es^2\Delta_x H_i^\es&=BCe_i\;\;\mbox{on}\;\;\Omega_\es,\label{aux4}\\
   \es \nabla H_i^\es\cdot n&= 0 \;\;\mbox{on}\;\;\Gamma_N^\es.\label{aux4bc}
\end{align}
Hence  \eqref{aux4} and \eqref{aux4bc} and following similar steps used to obtain  \eqref{st1}--\eqref{unbnd},  lead to
\begin{align}
    \dfrac{1}{2\es}\int_0^t\int_{\Omega^\es} B^\es C^\es (u^\es)^3  \nabla \psi_R^\es(t,x) dxdt\leq \dfrac{C}{R}.\label{stt}
\end{align}

\noindent Combining \eqref{st2}, \eqref{stini}, \eqref{st3}, \eqref{unbnd}, \eqref{stbdr} and \eqref{stt}, we get
\begin{equation}
    \lim_{R\rightarrow\infty }\int_{\Omega^\es} (u^\es(t,x))^2 \psi_R^\es(t,x)) dx=0.
\end{equation}
So, for any $\delta>0 $, there exists a $R_\delta$ such that
\begin{equation*}
    \int_{\Omega^\es} (u^\es(t,x))^2\psi_{R_{\delta}}^\es(t,x) ) dx\leq \delta.
\end{equation*}
Finally,  using the change of variable $x\rightarrow \left(x+\dfrac{B^*t}{\es}\right)$, we find
\begin{equation*}
    \int_{\Omega^\es(t)} \left(u^\es\left(t,x+\dfrac{B^*t}{\es}\right)\right)^2\psi_{R_{\delta}}(x) ) dx\leq \delta.
\end{equation*}
Hence, \eqref{comp} is proved.
\end{proof}

\begin{lemma}\label{lemma5}
Assume \ref{A1}--\ref{A5} hold. Then for $p,s\in (0,T)$
\begin{align}
   \left| \int_{\Omega^\es (p)}v^\es(p,x)q_{j,k}-\int_{\Omega^\es (s)}v^\es(s,x)q_{j,k}\right|\leq C\sqrt{p-s},\label{tt6}
\end{align}
where $C$ is a positive constant depending on $j$, $k\in \mathbb{N}$.
\end{lemma}
\begin{proof}
Using the Fundamental Theorem of Calculus, for $p,s\in (0,T)$ we can write
\begin{align}
    \int_{\Omega^\es (p)}v^\es(p,x)e_{j,k}-\int_{\Omega^\es (s)}v^\es(s,x)e_{j,k}=\int_{t}^{p}\dfrac{d}{dt} \int_{\Omega^\es (t)}v^\es(t,x)e_{j,k}.
\end{align}
By the change of variable $x\rightarrow x-\dfrac{B^*t}{\es}$ and the product rule, we get
\begin{align}
    \int_{\Omega^\es (p)}v^\es(p,x)e_{j,k}-\int_{\Omega^\es (s)}v^\es(s,x)e_{j,k}=\int_{s}^{p} \int_{\Omega^\es }\left(\dfrac{d}{dt}u^\es(t,x)q_{j,k}-\dfrac{B^*}{\es}\nabla q_{j,k} u^\es\right) dx dt.\label{tt1}
\end{align}
Choosing now $\phi=q_{j,k}$ as test function in \eqref{wf} and integrating the result from $s$ to $p$ with respect to time variable, we obtain
\begin{align}
      \int_{s}^{p}\int_{\Omega^\es} \partial_t u^\es q_{j,k} dx=-\int_{s}^{p}\int_{\Omega^\es} D^\es &\nabla u^\es \nabla q_{j,k} dxdt +\dfrac{1}{\es}\int_{s}^{p} \int_{\Omega^\es} B^\es u^\es \nabla q_{j,k} dxdt\nonumber\\
      &-\dfrac{1}{\es} \int_{s}^{p}\int_{\Omega^\es} B^\es C^\es (u^\es)^2 \nabla q_{j,k} dxdt+\int_{s}^{p}\int_{\Omega^\es} f^\es q_{j,k} dxdt\nonumber\\&\hspace{2cm}- \es\int_{s}^{p}\int_{{\Gamma_N^\es}} g_N(u^\es) q_{j,k} d\sigma dt.\label{tt2}
\end{align}
Combining \eqref{tt1} and \eqref{tt2} leads to
\begin{multline}\nonumber
    \int_{\Omega^\es (p)}v^\es(p,x)e_{j,k}-\int_{\Omega^\es (s)}v^\es(s,x)e_{j,k}=-\int_{s}^{p}\int_{\Omega^\es} D^\es \nabla u^\es \nabla q_{j,k} dxdt \\
      -\dfrac{1}{\es} \int_{s}^{p}\int_{\Omega^\es} B^\es C^\es (u^\es)^2 \nabla q_{j,k} dxdt+\int_{s}^{p}\int_{\Omega^\es} f^\es q_{j,k} dxdt\\- \es\int_{s}^{p}\int_{{\Gamma_N^\es}} g_N(u^\es) q_{j,k} d\sigma dt+\int_{s}^{p} \int_{\Omega^\es }\dfrac{B^\es-B^*}{\es}\nabla q_{j,k} u^\es dx dt.\label{tt3}
\end{multline}
Using Theorem \ref{energy}, \ref{A1}, and the definition of $q_{j,k}$, we have

\begin{align}
 \left|\int_{s}^{p}\int_{\Omega^\es} D^\es \nabla u^\es \nabla q_{j,k} dxdt\right|&\leq C \int_{s}^{p} \|\nabla u^\es\|_{L^2(\Omega^\es)} \|\nabla q_{j,k}\|_{L^2(\Omega^\es)}dt   \nonumber   \\ 
    &\leq C \|\nabla u^\es\|_{L^2(0,T;L^2(\Omega^\es))}\sqrt{p-s}  \\
   \label{tt5} &\leq C\sqrt{p-s},\nonumber
\end{align}
and
\begin{align}
    \left|\int_{s}^{p}\int_{\Omega^\es} f^\es q_{j,k} dxdt\right|\leq C\sqrt{p-s}.
\end{align}
The trace theorem and \ref{A4}  ensure that
\begin{align}
    \left| \es\int_{s}^{p}\int_{{\Gamma_N^\es}} g_N(u^\es) q_{j,k} d\sigma dt\right|&\leq  C\int_{s}^{p}\int_{{\Gamma_N^\es}\cap supp\{q_{j,k}\}} |u^\es||q_{j,k}| d\sigma dt\nonumber\\
       &\leq C \int_{s}^{p}\|u^\es\|_{L^2(\Gamma_N^\es)} dt\nonumber\\
       &\leq C \int_{s}^{p}\|\nabla u^\es\|_{L^2(\Omega^\es)} dt\nonumber\\
    &\leq C\sqrt{p-s}.
    \end{align}

Since the functions $q_{j,k}$ have compact support, using the auxiliary problem \eqref{aux2}--\eqref{aux2bc}, and the integration by parts with respect to space variable, we obtain

  \begin{align}
    \left|\int_{s}^{p} \int_{\Omega^\es }\dfrac{B^\es-B^*}{\es}\nabla q_{j,k} u^\es dx dt\right|&\leq \sum_{i=1}^{2} \left|\int_{s}^{p} \int_{\Omega^\es }\es \Delta G_i \frac{\partial}{\partial x_i} q_{j,k} u^\es dx dt\right|\nonumber\\
    &\leq \sum_{i=1}^{2} \left|\int_{s}^{p} \int_{\Omega^\es }\es \nabla G_i \nabla(\frac{\partial}{\partial x_i} q_{j,k}) u^\es dx dt\right|\nonumber\\ &+ \left|\int_{s}^{p} \int_{\Omega^\es }\es \nabla G_i \frac{\partial}{\partial x_i} q_{j,k}\nabla u^\es dx dt\right|\nonumber\\
    &\leq C\sqrt{p-s}.
\end{align}

Similarly, using auxiliary problem \eqref{aux4}--\eqref{aux4bc} and Theorem \ref{energy} we get
\begin{align}
    \left| \dfrac{1}{\es} \int_{s}^{p}\int_{\Omega^\es} B^\es C^\es (u^\es)^2 \nabla q_{j,k} dxdt\right|
    \leq C\sqrt{p-s}.\label{tt4}
    \end{align}

Combining the estimates \eqref{tt5}--\eqref{tt4} into \eqref{tt3}, we obtain \eqref{tt6}.

\end{proof}
\begin{theorem}\label{T5}
Assume \ref{A1}--\ref{A5} hold. Then there exists $v_0\in L^2(0,T;L^2(\mathbb{R}^2))$ such that
\begin{align}
    \lim_{\es\rightarrow 0}\int_{0}^{T}\int_{\Omega^{\es}(t)}\left|v^\es - v_0\right|^2dxdt=0,\label{strongm}
\end{align}
where $v^\es$ and $\Omega^\es(t)$ are defined in \eqref{ves} and \eqref{omees}, respectively. 
\end{theorem}

\begin{proof}
We first prove that the sequence $\{v^\es_{jk}\}$ has a subsequence that converges to some $\{v^0_{jk}\}$ as $\es\rightarrow 0$. From \eqref{vjk},  we have
\begin{align}
    |v^\es_{jk}(t)|&\leq \int_{\Omega^\es}|v^\es(t,x)e_{jk}(x)|dx\nonumber\\
    &\leq \|v^\es\|_{L^2(\Omega^\es)}\|e_{jk}\|_{L^2(\Omega^\es)}\nonumber\\
    &\leq C.
    \end{align}
    By Lemma \ref{lemma5}, we know that the sequence $\{v^\es_{jk}(t)\}$ is equicontinous, hence by using Arzela-Ascoli Theorem we have that there exists a subsequence (denoted again by $v^\es_{jk}(t)$) such that $v^\es_{jk}(t)$ uniformly converges to some $v^0_{jk}(t)$ on $C([0,T])$. Since $[0,T]$ is a bounded domain, the uniform convergence leads to strong convergence in $L^2(0,T)$.
    We define 
    \begin{align}
    v^0_*(t,x):=\sum_{j\in \mathbb{N},k\in \mathbb{Z}^2}v^0_{jk}(t)e_{jk}(x).\label{t5e10}
\end{align}
\textit{Claim 1}: for fixed $j$ and $k$, there exists a constant $C_{jk}>0$ such that
\begin{align}
    \left|v^\es_{jk}(t)-\dfrac{|Y|}{|Z|}\Tilde{v}^\es_{jk}(t)\right|\leq C_{jk}\es ^2\label{t5e7}.
\end{align}
Proof of \textit{Claim 1}: 
\begin{align}
    v^\es_{jk}(t)-\dfrac{|Y|}{|Z|}\Tilde{v}^\es_{jk}(t)=\int_{\mathbb{R}^2}\left(\chi_{Y}(\frac{x}{\es})-\dfrac{|Y|}{|Z|}\right)\Tilde{v}^\es(t,x)e_{jk}(x)dx,\label{t5e1}
\end{align}
where $\chi_{Y}$ is the characteristic function defined on $Z$ and extended periodically to whole $\mathbb{R}^2$, where $\chi_{Y}(x)=1$ if $x\in Y$, $0$ if $x\in Y_0$. Since $\int_Z\left(\chi_{Y}-\dfrac{|Y|}{|Z|}\right)dx=0$, the following auxiliary problem has a unique weak solution
\begin{align}
    &-\Delta \Pi(y) = \chi_{Y}(y)-\dfrac{|Y|}{|Z|}\hspace{1cm}\mbox{on} \;\;Z,\\
    &\Pi-Z \mbox{periodic}.
\end{align}
Using the change of variable $y=\frac{x}{\es}$, defining $\Pi^\es(x):=\Pi(\frac{x}{\es})$, and stating the problem for whole $\mathbb{R}^2$, we get
\begin{align}
    &-\es^2\Delta \Pi^\es(x) = \chi_{Y}(\frac{x}{\es})-\dfrac{|Y|}{|Z|}\hspace{1cm}\mbox{on} \;\;\mathbb{R}^2\label{t5e2}.
\end{align}
We use \eqref{t5e2} in \eqref{t5e1} and we get
\begin{align}
    \left|v^\es_{jk}(t)-\dfrac{|Y|}{|Z|}\Tilde{v}^\es_{jk}(t)\right|&\leq \es^2\int_{\mathbb{R}^2}\left|\nabla \Pi^\es\nabla(\Tilde{v}^\es(t,x)e_{jk}(x))\right|dx\nonumber\\
    &\leq \es^2\int_{\mathbb{R}^2\cap supp(e_{jk})}\left|\nabla \Pi^\es\nabla(\Tilde{v}^\es(t,x)e_{jk}(x))\right|dx\nonumber\\
    &\leq C_{jk}\es^2\label{t5e3}.
\end{align}

Note that to get the inequality \eqref{t5e3}, we used
\eqref{ext3} and Theorem \ref{energy}.\\ 
\noindent Hence we proved the \textit{Claim 1}.
\par \textit{Claim 2}: for any $\delta>0$ there exists $R_\delta\in \mathbb{N}$ such that 
\begin{align}
    \left\|\Tilde{ v}^\es\chi_{[-R_\delta,R_\delta]^2}-\sum_{|j|,|k|\leq R_\delta}\Tilde{v}^\es_{jk}e_{jk} \right\|_{L^2(0,T;L^2(\mathbb{R}^2))}< \frac{\delta}{5},\label{t5e5}
\end{align}
where $\chi_{[-R_\delta,R_\delta]^2}$ is characteristic function of $[-R_\delta,R_\delta]^2$.\\
Proof of \textit{Claim 2} follows similar arguments as in  Lemma 4 from \cite{marusik2005}. We use Theorem \ref{energy},  Lemma \ref{lemma4}, Lemma \ref{extension} and Rellich–Kondrachov theorem. For details see Lemma 4 from \cite{marusik2005}.
\par Now, from Lemma \ref{lemma4} there exists $N_1\in \mathbb{N}$ such that
\begin{align}
    \left\|{ v}^\es-{ v}^\es\chi_{[-N_1,N_1]^2} \right\|_{L^2(0,T;L^2(\mathbb{R}^2))}< \frac{\delta}{5}\label{t5e12}.
\end{align}
Using the property \eqref{ext1}  for $\Tilde{v}^\es$ and \eqref{t5e5}, we can guarantee that 
\begin{align}
    \left\|{ v}^\es\chi_{[-R_\delta,R_\delta]^2}-\sum_{|j|,|k|\leq R_\delta}\Tilde{v}^\es_{jk}e_{jk} \right\|_{L^2(0,T;L^2(\Omega^{\es}(t)))}< \frac{\delta}{5}\label{t5e6}.
\end{align}
Choosing $\es$ small enough and using \eqref{t5e7}, we have 
\begin{align}
    \left\|\sum_{|j|,|k|\leq R_\delta}\Tilde{v}^\es_{jk}e_{jk}-\dfrac{|Z|}{|Y|}\sum_{|j|,|k|\leq R_\delta}{v}^\es_{jk}e_{jk} \right\|_{L^2(0,T;L^2(\Omega^{\es}(t)))}< \frac{\delta}{5}.\label{t5e8}
\end{align}
Since ${v}^\es_{jk}$ strongly converges to ${v}^0_{jk}$ in $L^2(0,T)$, for small enough $\es$, we are lead to
\begin{align}
    \left\|\sum_{|j|,|k|\leq R_\delta}{v}^\es_{jk}e_{jk}-\sum_{|j|,|k|\leq R_\delta}{v}^0_{jk}e_{jk} \right\|_{L^2(0,T;L^2(\Omega^{\es}(t)))}< \dfrac{|Z|}{|Y|}\frac{\delta}{5}.\label{t5e9}
\end{align}

From \eqref{t5e10}, we get
\begin{align}
    \left\|\sum_{|j|,|k|\leq R_\delta}{v}^0_{jk}e_{jk}-v^0_*(t,x) \right\|_{L^2(0,T;L^2(\Omega^{\es}(t))}< \frac{\delta}{5}.\label{t5e11}
\end{align}
Choosing in  \eqref{t5e12}--\eqref{t5e11} $N_1,R_\delta$ large enough, we obtain for  small enough  $\es$ that 
\begin{align}
    \left\|v^\es(t,x)-\dfrac{|Z|}{|Y|}v^0_*(t,x) \right\|_{L^2(0,T;L^2(\Omega^{\es}(t))}< \frac{\delta}{5}.\label{t5e13}
\end{align}
Finally, we conclude the proof by choosing $v^0(t,x)=\dfrac{|Z|}{|Y|}v^0_*(t,x).$
\end{proof}

{
    Observe that, by using the change of variable $x\rightarrow x-\dfrac{B^*t}{\es}$ into the identity \eqref{strongm}, we obtain the following strong two-scale convergence with drift 
    \begin{align}
        \lim_{\es\rightarrow 0}\int_{0}^{T}\int_{\Omega^{\es}}\left|u^\es(t,x) - v_0(t,x-\dfrac{B^*t}{\es})\right|^2dxdt=0.\label{strongc}
    \end{align}
    This is a useful result that will appear in a several context in the following.
    }

\subsection{Two-scale convergence with drift}
In this section we recall the concept of two-scale convergence with drift.

 \begin{definition}[Two scale convergence with drift r]\label{def2scd}
 Let $r\in \mathbb{R}^2$ and $u^\es\in L^2(0,T;L^2(\Omega^\es)$, we say $u^\es$ two-scale converges with drift $r$ to $u_0$, if for all $\phi\in C_c^\infty((0,T)\times\mathbb{R}^2;C_{\#}^\infty (Z))$ the following identity holds
\begin{align}
    \lim_{\es\rightarrow 0}\int_{(0,T)\times\Omega^\es} u^\es(t,x)\phi(t,x-\dfrac{rx}{\es},\frac{x}{\es})dxdt= \int_{0}^{T}\int_{\mathbb{R}^2}\int_Z u_0(t,x,y)\phi(t,x,y)dydxdt\label{tdriftdef}.
\end{align}
 We denote the convergence as $u^\es\xrightharpoonupdown{$2-drift(r)$}u_0$.
 \end{definition}
 \begin{definition}[Strong two-scale convergence with drift]\label{defs2scd}
        Let $r\in \mathbb{R}^2$ and \\ $u^\es\in L^2(0,T;L^2(\Omega^\es)$. We say that $u^\es$ strongly two-scale converges with the drift $r$ to $u_0$ if and only if  
        \begin{align}
        \lim_{\es\rightarrow 0}\int_{0}^{T}\int_{\Omega^{\es}}\left|u^\es(t,x) - u_0\left(t,x-\dfrac{rt}{\es},\dfrac{x}{\es}\right)\right|^2dxdt=0.\label{strongc2}
    \end{align}
    We denote this convergence by $u^\es\xrightarrow{2-drift(r)}u_0$.
    \end{definition}
 
\subsubsection{Compactness results}
\begin{lemma}\label{oscilationlemma}
 Let $\phi\in L^2((0,T)\times \mathbb{R}^2;C_{\#}(Z))$. Then 
    \begin{align}
        \lim_{\es\rightarrow 0} \int_0^T\int_{\mathbb{R}^2}\phi(t,x-\dfrac{B^*t}{\es},\dfrac{x}{\es})^2dxdt=\int_0^T\int_{\mathbb{R}^2}\int_Z\phi(t,x,y)^2dydxdt
    \end{align}
    holds true.
\end{lemma}
\begin{proof}
We refer the reader to Proposition 2.6.7 in \cite{hutridur} for the details leading to a proof of this statement.
\end{proof}
\begin{theorem}\label{compactness}
Let $u^\es\in L^2(0,T;H^1(\Omega^\es))$. Assume there exists a constant $C>0$ independent of $\es$ such that
\begin{align}
    \|u^\es\|_{L^2(0,T;H^1(\Omega^\es))}\leq C.
\end{align}
Then, there exist $u_0\in L^2(0,T;H^1(\mathbb{R}^2))$ and $u_1\in L^2((0,T)\times H^1(\mathbb{R}^2);H_{\#}^1(Z)) $ such that
\begin{align}
    u^\es \xrightharpoonupdown{$2-drift(r)$} u_0,\\
    \nabla u^\es \xrightharpoonupdown{$2-drift(r)$} \nabla_x u_0+\nabla_y u_1.
\end{align}
\end{theorem}
\begin{proof}
For a detailed proof of this compactness result,  see \cite{marusik2005}.
\end{proof}
\begin{theorem}\label{compactnessb}
Let $u^\es\in L^2(0,T;L^2(\Gamma_N^\es))$. Assume there exists a constant $C>0$ independent of $\es$ such that
\begin{align}
   \es \|u^\es\|_{L^2(0,T;L^2(\Gamma_N^\es))}\leq C.
\end{align}
Then, there exists $u_0\in L^2(0,T;L^2(\mathbb{R}^2\times \Gamma_N))$  such that
\begin{align}
   \lim_{\es\rightarrow 0}\es \int_{(0,T)\times\Gamma_N^\es} u^\es(t,x)\phi(t,x-\dfrac{rt}{\es},\frac{x}{\es})dxdt= \int_{0}^{T}\int_{\mathbb{R}^2}\int_{\Gamma_N} u_0(t,x,y)\phi(t,x,y)dydxdt.
\end{align}
\end{theorem}
\begin{proof}
For a detailed proof, see Proposition 5.4 of \cite{hutridurga}.
\end{proof}
\par 
 It is useful to 
    note that the function $v_0$ which we obtained from \eqref{strongc} and $u_0$ coming from \eqref{2drift1} are equal for a.e $(t,x)\in (0,T)\times \mathbb{R}^2$.  To see this, we can argue as follows: 
    for any $\phi\in C_c^\infty((0,T)\times \mathbb{R}^2)$, we have
    \begin{align}
       \int_0^T\int_{\mathbb{R}^2}(u_0-v_0)\phi dxdt&=\lim_{\es\rightarrow 0} \frac{1}{|Z|}\int_0^T\int_{\Omega^\es}\left(u^\es(t,x)-v_0(t,x-\dfrac{B^*t}{\es})\right)\phi(t,x-\dfrac{B^*t}{\es})dxdt\nonumber\\
       &\leq \lim_{\es\rightarrow 0} \frac{1}{|Z|}\|u^\es(t,x)-v_0(t,x,x-\dfrac{B^*t}{\es})\|_{L^2((0,T)\times \Omega^\es)}\|\phi(t,x-\dfrac{B^*t}{\es})\|_{L^2((0,T)\times \Omega^\es)}\nonumber\\
       &=0\nonumber.
    \end{align}
    Hence, we can conclude that $v_0$ and $u_0$ coincide.

\subsection{Limit problem -- Structure of the upscaled model equations}
The main result of this paper is stated in the next Theorem. A connected companion result is Theorem \ref{Cor}.
\begin{theorem}\label{Homlimit}
Assume \ref{A1}--\ref{A5} hold { and $g_N(r)=r $ for all $r\in\mathbb{R}$}. Then the weak solution $u^\es$ to the microscopic problem \eqref{meq}--\eqref{meqic} two-scale converges with the drift $B^*$  to $u^0(t,x)$  as $\es\rightarrow 0$, where $u^0(t,x)$ is the weak solution of the homogenized problem, viz.
\begin{tcolorbox}
\begin{align}
    \partial_t u_0 +\mathrm{div}( -D^*(u_0,W)\nabla_x  u_0)&=\frac{1}{|Z|}\int_Z f\, dy - {\frac{|\Gamma_N|}{|Z|}g_N(u_0) } &\mbox{on}& \hspace{.4cm}  (0,T)\times \mathbb{R}^2, \label{homeq1}\\
    u_0(0)&=g &\mbox{on}& \hspace{1.75cm}  \mathbb{R}^2
    \end{align}
  \begin{align}
    -\nabla_y\cdot D(y) \nabla_y w_i+B(y)(1-2C(y)u_0)\cdot\nabla_y  w_i&=\nabla_y D(y) e_i+B^* e_i-B(y)(1-2C(y)u_0)\cdot  e_i \nonumber\\
    &\hspace{3.5cm}\mbox{on} \hspace{.4cm}  (0,T)\times \mathbb{R}^2\times Z,\label{cell3}\\
    \left( -D(y)\nabla_y w_i+B(y)(1-2C(y)u_0)w_i\right)\cdot n_y &= \left( -D(y)e_i\right)\cdot n_y\nonumber \\
    &\hspace{3.5cm}\mbox{on} \hspace{.4cm} (0,T)\times \mathbb{R}^2\times \Gamma_N,\label{cellbn3}\\
    w_i(t,x,\cdot)\, &\mbox{ is $Z$--periodic},\label{homeqf}
\end{align}
\end{tcolorbox}
where 
\begin{tcolorbox}
\begin{equation}
    B^*e_i:=\dfrac{\int_Z B(y)e_idy}{|Z|}
\end{equation}
\end{tcolorbox}
and 
\begin{tcolorbox}
\begin{multline}\label{effectdiff}
     D^*(u_0,W):=\frac{1}{|Z|}\int_Z D(y)\left(I+\begin{bmatrix}
\frac{\partial w_1}{\partial y_1} & \frac{\partial w_2}{\partial y_1} \\
\frac{\partial w_1}{\partial y_2} & \frac{\partial w_2}{\partial y_2}
\end{bmatrix}\right)\, dy
\\
+\frac{1}{|Z|}B^*\int_Z W(y)^t\, dy-\frac{1}{|Z|}\int_ZB(y)(1-2C(y)u_0) W(y)^t\, dy
 \end{multline}
 \end{tcolorbox}
 for all $t\in (0,T)$ and a.e $x\in \mathbb{R}^2, y\in Z $, with
 \begin{equation}
     W:=(w_1,w_2).
 \end{equation}
\end{theorem}
\begin{proof}
Using the compactness result stated in Theorem \ref{compactness} and the energy estimates obtained in Theorem \ref{energy}, we can state that there exist $u_0\in L^2(0,T;H^1(\mathbb{R}^2))$ and $u_1\in L^2((0,T)\times H^1(\mathbb{R}^2);H_{\#}^1(Z)) $ such that
\begin{align}
    u^\es \xrightharpoonupdown{$2-drift(B^*)$} u_0\label{2drift1},\\
    \nabla u^\es \xrightharpoonupdown{$2-drift(B^*)$} \nabla_x u_0+\nabla_y u_1.\label{2drift2}
\end{align}
Take $\phi_0\in C_c^\infty((0,T)\times\mathbb{R}^2)$ and $\phi_1\in C_c^\infty((0,T)\times\mathbb{R}^2)\times C_{\#}^\infty(Z)$. Now, in the weak formulation \eqref{weakformul} of our microscopic problem, we choose $\phi=\phi_0(t,x-\dfrac{B^*t}{\es})+\es \phi_1(t,x-\dfrac{B^*t}{\es},\dfrac{x}{\es})$ to obtain: 
\begin{multline}\label{t7e1}
     \int_0^T\int_{\Omega^\es} \partial_t u^\es \phi_0(t,x-\dfrac{B^*t}{\es}) dxdt+\int_0^T\int_{\Omega^\es} D^\es \nabla u^\es \nabla \phi_0(t,x-\dfrac{B^*t}{\es}) dxdt \\-\dfrac{1}{\es}\int_0^T\int_{\Omega^\es} B^\es u^\es(1-C^\es u^\es) \nabla \phi_0(t,x-\dfrac{B^*t}{\es}) dxdt
     +\es\int_0^T\int_{\Omega^\es} \partial_t u^\es \phi_1(t,x-\dfrac{B^*t}{\es},\dfrac{x}{\es}) dxdt
    \\
    +\es\int_0^T\int_{\Omega^\es} D^\es \nabla u^\es \nabla \phi_1(t,x-\dfrac{B^*t}{\es},\dfrac{x}{\es}) dxdt 
     - \int_0^T\int_{\Omega^\es} B^\es u^\es(1-C^\es u^\es) \nabla \phi_1(t,x-\dfrac{B^*t}{\es},\dfrac{x}{\es}) dxdt
    \\
    =  \int_0^T\int_{\Omega^\es} f^\es \phi_0(t,x-\dfrac{B^*t}{\es}) dxdt- \es\int_0^T\int_{{\Gamma_N^\es}} g_N(u^\es) \phi_0(t,x-\dfrac{B^*t}{\es}) d\sigma dt\\
      \es\int_0^T\int_{\Omega^\es} f^\es \phi_1(t,x-\dfrac{B^*t}{\es},\dfrac{x}{\es}) dx- \es^2\int_0^T\int_{{\Gamma_N^\es}} g_N(u^\es) \phi_1(t,x-\dfrac{B^*t}{\es},\dfrac{x}{\es})d\sigma dt.
\end{multline}
By using the integration by parts and the chain rule, we have 
\begin{multline}\label{t7e2}
    \int_0^T\int_{\Omega^\es} \partial_t u^\es \phi_0(t,x-\dfrac{B^*t}{\es}) dxdt\\
    =-\int_0^T\int_{\Omega^\es}  u^\es \partial_t\phi_0(t,x-\dfrac{B^*t}{\es}) dxdt+ \frac{B^*}{\es}\int_0^T\int_{\Omega^\es}  u^\es \nabla_x\phi_0(t,x-\dfrac{B^*t}{\es}) dxdt.
\end{multline}
We rewrite the next term as
\begin{multline}\label{t7e3}
     \es\int_0^T\int_{\Omega^\es} \partial_t u^\es \phi_1(t,x-\dfrac{B^*t}{\es},\dfrac{x}{\es}) dx\\=-\es\int_0^T\int_{\Omega^\es}  u^\es \partial_t\phi_1(t,x-\dfrac{B^*t}{\es},\dfrac{x}{\es}) dxdt+  B^*\int_{\Omega^\es}  u^\es \nabla_x\phi_1(t,x-\dfrac{B^*t}{\es},\dfrac{x}{\es}) dxdt.\\
\end{multline}
The following calculations rules will be used in the sequel, viz.
\begin{align}
    \nabla(\phi_1(t,x-\dfrac{B^*t}{\es},\dfrac{x}{\es}))&=\nabla_x\phi_1(t,x-\dfrac{B^*t}{\es},\dfrac{x}{\es})+\frac{1}{\es}\nabla_y\phi_1(t,x-\dfrac{B^*t}{\es},\dfrac{x}{\es}),\label{t7e4}\\
    \nabla(\phi_0(t,x-\dfrac{B^*t}{\es}))&=\nabla_x\phi_0(t,x-\dfrac{B^*t}{\es}).
\end{align}
From assumption \ref{A2}, we deduce 
\begin{multline}\label{t7e5}
    \int_0^T\int_{\Omega^\es} B^\es u^\es(1-C^\es u^\es) \nabla \phi_0(t,x-\dfrac{B^*t}{\es}) dxdt\\=-\int_0^T\int_{\Omega^\es} B^\es \nabla u^\es \phi_0(t,x-\dfrac{B^*t}{\es}) dxdt+\int_0^T\int_{\Omega^\es} B^\es C^\es u^\es \nabla u^\es \phi_0(t,x-\dfrac{B^*t}{\es}) dxdt.
\end{multline}
To derive the structure of the cell problem,  we choose in \eqref{t7e1} as test function $\phi_0\equiv 0$. Then from \eqref{t7e2}--\eqref{t7e5}, we are led to
\begin{multline}\label{t7e6}
    -\es\int_0^T\int_{\Omega^\es}  u^\es \partial_t\phi_1(t,x-\dfrac{B^*t}{\es},\dfrac{x}{\es}) dxdt+  B^*\int_0^T\int_{\Omega^\es}  u^\es \nabla_x\phi_1(t,x-\dfrac{B^*t}{\es},\dfrac{x}{\es}) dxdt\\
    +\es\int_0^T\int_{\Omega^\es} D^\es \nabla u^\es \nabla_x \phi_1(t,x-\dfrac{B^*t}{\es},\dfrac{x}{\es}) dxdt+\int_0^T\int_{\Omega^\es} D^\es \nabla u^\es \nabla_y \phi_1(t,x-\dfrac{B^*t}{\es},\dfrac{x}{\es}) dxdt\\
    \int_0^T\int_{\Omega^\es} B^\es(1-2C^\es u^\es)\nabla u^\es  \phi_1(t,x-\dfrac{B^*t}{\es},\dfrac{x}{\es}) dxdt=\es \int_0^T\int_{\Omega^\es}f^\es \phi_1(t,x-\dfrac{B^*t}{\es},\dfrac{x}{\es})\\
    -\es ^2 \int_0^T\int_{\Gamma_N^\es}g_N(u^\es) \phi_1(t,x-\dfrac{B^*t}{\es},\dfrac{x}{\es})d\sigma dt.
\end{multline}
As a direct application of Theorem \ref{energy}, we obtain
\begin{align}
    \lim_{\es\rightarrow 0}\es\int_0^T\int_{\Omega^\es}  u^\es \partial_t\phi_1(t,x-\dfrac{B^*t}{\es},\dfrac{x}{\es}) dxdt=0.\label{t7e7}
\end{align}
Using \eqref{2drift1}, we have
\begin{align}
   \lim_{\es\rightarrow 0} B^*\int_0^T\int_{\Omega^\es}  u^\es \nabla_x\phi_1(t,x-\dfrac{B^*t}{\es},\dfrac{x}{\es}) dxdt= &B^*\int_0^T\int_{\mathbb{R}^2}\int_Z  u_0 \nabla_x\phi_1(t,x,y) dydxdt\nonumber
    \\
    &=-B^*\int_0^T\int_{\mathbb{R}^2}\int_Z  \nabla_xu_0 \phi_1(t,x,y) dxdt.\label{t7e8}
\end{align}
Relying on \ref{A1} and \eqref{2drift2}, we can write

\begin{align}
  \lim_{\es\rightarrow 0} \es\int_0^T\int_{\Omega^\es} D^\es \nabla u^\es \nabla_x \phi_1(t,x-\dfrac{B^*t}{\es},\dfrac{x}{\es}) dxdt&= 0, \\
   \lim_{\es\rightarrow 0}\int_0^T\int_{\Omega^\es} D^\es \nabla u^\es \nabla_y \phi_1(t,x-\dfrac{B^*t}{\es},\dfrac{x}{\es}) dxdt&= \int_0^T\int_{\mathbb{R}^2}\int_Z D^\es (\nabla_xu_0+\nabla_yu_1) \nabla_y \phi_1(t,x,y) dydxdt. \label{t7e9}
\end{align}
Using \ref{A2} together with the  strong convergence result stated in Theorem \ref{T5}, and recalling as well \eqref{2drift2}, we get
\begin{multline}\label{t7e10}
   \lim_{\es\rightarrow 0}\int_0^T\int_{\Omega^\es} B^\es(1-2C^\es u^\es)\nabla u^\es  \phi_1(t,x-\dfrac{B^*t}{\es},\dfrac{x}{\es}) dxdt\\= \int_0^T\int_{\mathbb{R}^2}\int_Z B(y)(1-2C(y) u_0)(\nabla_x u_0+\nabla_yu_1) \phi_1(t,x,y) dydxdt.
\end{multline}
By \ref{A3} and Theorem \ref{energy}, we get
\begin{align}
   \lim_{\es\rightarrow 0}\es \int_0^T\int_{\Omega^\es} f^\es \phi_1(t,x-\dfrac{B^*t}{\es},\dfrac{x}{\es}) dxdt= 0.\label{t7e11}
\end{align}
Using \ref{A4} jointly with Theorem \ref{compactnessb} yields
\begin{align}
     \lim_{\es\rightarrow 0}\es^2\int_0^T\int_{{\Gamma_N^\es}} g_N(u^\es) \phi_1(t,x-\dfrac{B^*t}{\es},\dfrac{x}{\es})d\sigma dt= 0.\label{t7e12}
\end{align}
Now, passing to $\es\rightarrow 0$ in \eqref{t7e6} and using \eqref{t7e7}--\eqref{t7e12}, we obtain
\begin{multline}\label{t7e13}
    -B^*\int_0^T\int_{\mathbb{R}^2}\int_Z  \nabla_xu_0 \phi_1(t,x,y) dydxdt+ \int_0^T\int_{\mathbb{R}^2}\int_Z D^\es \nabla_xu_0 \nabla_y \phi_1(t,x,y) dydxdt\\+\int_0^T\int_{\mathbb{R}^2}\int_Z D^\es \nabla_yu_1 \nabla_y \phi_1(t,x,y) dydxdt+\int_0^T\int_{\mathbb{R}^2}\int_Z B(y)(1-2C(y) u_0)\nabla_x u_0 \phi_1(t,x,y) dydxdt\\
    +\int_0^T\int_{\mathbb{R}^2}\int_Z B(y)(1-2C(y) u_0)\nabla_yu_1 \phi_1(t,x,y) dydxdt=0.
\end{multline}
Looking now at \eqref{t7e13}, the  choice $\phi_1(t,x,y)=\phi_2(t,x)\phi_3(y)$ yields for almost every $(t,x)\in (0,T)\times \mathbb{R}^2$ the identity:
\begin{multline}\label{t7e14}
    -B^*\int_Z  \nabla_xu_0 \phi_3(y) dy+ \int_Z D^\es \nabla_xu_0 \nabla_y \phi_3(y) dy+\int_Z D^\es \nabla_yu_1 \nabla_y \phi_3(y) dy\\+\int_Z B(y)(1-2C(y) u_0)\nabla_x u_0 \phi_3(y) dy
    +\int_Z B(y)(1-2C(y) u_0)\nabla_yu_1 \phi_3(y) dy=0.
\end{multline}
The structure of \eqref{t7e14} allows us to choose further 
\begin{align}\label{cell}
u_1(t,x,y):=W(t,x,y)\cdot \nabla_x u_0(x),
\end{align}
where $W:=(w_1,w_2)$ with $w_i$ (with $i\in \{1,2\}$) solving the following cell problem:
\begin{multline}\label{t7e15}
    -B^*\int_Z  e_i \phi_3(y) dy+ \int_Z D^\es e_i \nabla_y \phi_3(y) dy+\int_Z D^\es \nabla_yw_i \nabla_y \phi_3(y) dy\\+\int_Z B(y)(1-2C(y) u_0) e_i \phi_3(y) dy
    +\int_Z B(y)(1-2C(y) u_0)\nabla_yw_i \phi_3(y) dy=0.
\end{multline}

Note that \eqref{t7e15} is the weak formulation of the cell problem \eqref{cell3}--\eqref{homeqf}.
\par To derive the macroscpic equation \eqref{homeq1}, in \eqref{t7e1} we choose $\phi_1\equiv 0$  and we get
\begin{multline}\label{t7e16}
     -\int_0^T\int_{\Omega^\es}  u^\es \partial_t\phi_0(t,x-\dfrac{B^*t}{\es}) dxdt +\int_0^T\int_{\Omega^\es} D^\es \nabla u^\es \nabla \phi_0(t,x-\dfrac{B^*t}{\es}) dxdt\\     + \int_0^T\int_{\Omega^\es}\frac{B^*-B^\es}{\es}  u^\es \nabla_x\phi_0(t,x-\dfrac{B^*t}{\es}) dxdt +\dfrac{1}{\es} \int_{\Omega^\es} B^\es C^\es (u^\es)^2 \nabla \phi_0(t,x-\dfrac{B^*t}{\es}) dxdt\\
    =  \int_0^T\int_{\Omega^\es} f^\es \phi_0(t,x-\dfrac{B^*t}{\es}) dxdt- \es\int_0^T\int_{{\Gamma_N^\es}} g_N(u^\es) \phi_0(t,x-\dfrac{B^*t}{\es}) d\sigma dt.
\end{multline}
Using \eqref{2drift1}, we get
\begin{align}
  \lim_{\es\rightarrow 0} -\int_0^T\int_{\Omega^\es}  u^\es \partial_t\phi_0(t,x-\dfrac{B^*t}{\es}) dxdt&= -\int_0^T\int_{\mathbb{R}^2}\int_Z  u_0 \partial_t\phi_0(t,x) dydxdt\nonumber \\
   &=|Z|\int_0^T\int_{\mathbb{R}^2}\partial_t u_0 \phi_0(t,x) dxdt.\label{t7e17}
\end{align}
Using \eqref{2drift2} and \ref{A1}, we get
\begin{align}
   \lim_{\es\rightarrow 0}\int_0^T\int_{\Omega^\es} D^\es \nabla u^\es \nabla_x \phi_0(t,x-\dfrac{B^*t}{\es}) dxdt&= \int_0^T\int_{\mathbb{R}^2}\int_Z  D(y) (\nabla_x u^0+\nabla_y u_1) \nabla_x \phi_0(t,x) dxdt.\label{t7e18}
\end{align}
Using the auxiliary problem \eqref{aux2}--\eqref{aux2bc} and \eqref{2drift1}, \eqref{2drift2} we have 
\begin{align}
  \lim_{\es\rightarrow 0}\int_0^T\int_{\Omega^\es}&\frac{B^*-B^\es}{\es}  u^\es \nabla_x\phi_0(t,x-\dfrac{B^*t}{\es}) dxdt=\lim_{\es\rightarrow 0}-\sum_{i=1}^{2}\int_0^T\int_{\Omega^\es}\es \Delta G_i^\es  u^\es \frac{\partial}{\partial x_i}\phi_0(t,x-\dfrac{B^*t}{\es}) dxdt\nonumber\\
  &=\lim_{\es\rightarrow 0}\es\sum_{i=1}^{2}\int_0^T\int_{\Omega^\es} \nabla G_i^\es  \nabla (u^\es \frac{\partial}{\partial x_i}\phi_0(t,x-\dfrac{B^*t}{\es})) dxdt\nonumber\\
  &=\lim_{\es\rightarrow 0}\es\sum_{i=1}^{2}\int_0^T\int_{\Omega^\es}\frac{1}{\es} \nabla_y G_i(\frac{x}{\es})  \nabla u^\es \frac{\partial}{\partial x_i}\phi_0(t,x-\dfrac{B^*t}{\es}) dxdt\nonumber\\
  &\hspace{3cm}+\lim_{\es\rightarrow 0}\es\sum_{i=1}^{2}\int_0^T\int_{\Omega^\es}\frac{1}{\es} \nabla_y G_i(\frac{x}{\es})   u^\es \nabla_x(\frac{\partial}{\partial x_i}\phi_0(t,x-\dfrac{B^*t}{\es})) dxdt\nonumber\\
  &= \sum_{i=1}^{2}\int_0^T\int_{\mathbb{R}^2}\int_Z \nabla_yG_i(y)(\nabla_xu_0+\nabla_yu_1)\frac{\partial \phi_0}{\partial x_i}(t,x)dydxdt\nonumber\\
  &\hspace{3cm}+\sum_{i=1}^{2}\int_0^T\int_{\mathbb{R}^2}\int_Z \nabla_yG_i(y)u_0\nabla_x\left(\frac{\partial \phi_0}{\partial x_i}(t,x)\right)dydxdt\nonumber\\
  &=\sum_{i=1}^{2}\int_0^T\int_{\mathbb{R}^2}\int_Z \nabla_yG_i(y)(\nabla_xu_0+\nabla_yu_1)\frac{\partial \phi_0}{\partial x_i}(t,x)dydxdt\nonumber\\
  &\hspace{3cm}-\sum_{i=1}^{2}\int_0^T\int_{\mathbb{R}^2}\int_Z \nabla_yG_i(y)\nabla_xu_0\left(\frac{\partial \phi_0}{\partial x_i}(t,x)\right)dydxdt\nonumber\\
  &=\sum_{i=1}^{2}\int_0^T\int_{\mathbb{R}^2}\int_Z \nabla_yG_i(y)\nabla_yu_1\frac{\partial \phi_0}{\partial x_i}(t,x)dydxdt\nonumber\\
  &=-\sum_{i=1}^{2}\int_0^T\int_{\mathbb{R}^2}\int_Z \Delta_yG_i(y)u_1\frac{\partial \phi_0}{\partial x_i}(t,x)dydxdt\nonumber\\
  &=\int_0^T\int_{\mathbb{R}^2}\int_Z(B^*-B(y))u_1\nabla_x\phi_0(t,x)dydxdt\nonumber\\
  &=-\int_0^T\int_{\mathbb{R}^2}\int_Z(B^*-B(y))\nabla_x u_1\phi_0(t,x)dydxdt.\label{T7E18}
\end{align}
Using the auxiliary problem \eqref{aux4}--\eqref{aux4bc}, \eqref{2drift1} and \eqref{2drift2}, we obtain 
\begin{align}
  \lim_{\es\rightarrow 0}\int_0^T\int_{\Omega^\es}&\frac{B^\es C^\es}{\es}  (u^\es)^2 \nabla_x\phi_0(t,x-\dfrac{B^*t}{\es}) dxdt=-\lim_{\es\rightarrow 0}\sum_{i=1}^{2}\int_0^T\int_{\Omega^\es}\es \Delta H_i^\es  (u^\es)^2 \frac{\partial}{\partial x_i}\phi_0(t,x-\dfrac{B^*t}{\es}) dxdt\nonumber\\
  &=\lim_{\es\rightarrow 0}\es\sum_{i=1}^{2}\int_0^T\int_{\Omega^\es} \nabla H_i^\es  \nabla ((u^\es)^2 \frac{\partial \phi_0}{\partial x_i}(t,x-\dfrac{B^*t}{\es})) dxdt\nonumber\\
  &={2}\lim_{\es\rightarrow 0}\es\sum_{i=1}^{2}\int_0^T\int_{\Omega^\es}\frac{1}{\es} \nabla_y H_i(\frac{x}{\es})  u^\es \nabla u^\es \frac{\partial \phi_0}{\partial x_i}(t,x-\dfrac{B^*t}{\es}) dxdt\nonumber\\
  &\hspace{3cm}+\lim_{\es\rightarrow 0}\es\sum_{i=1}^{2}\int_0^T\int_{\Omega^\es}\frac{1}{\es} \nabla_y H_i(\frac{x}{\es})   (u^\es)^2 \nabla_x(\frac{\partial \phi_0}{\partial x_i}(t,x-\dfrac{B^*t}{\es})) dxdt\nonumber\\
  = &{2}\sum_{i=1}^{2}\int_0^T\int_{\mathbb{R}^2}\int_Z \nabla_yH_i(y)u_0(\nabla_xu_0+\nabla_yu_1)\frac{\partial \phi_0}{\partial x_i}(t,x)dydxdt \nonumber\\
  &\hspace{3cm}+\sum_{i=1}^{2}\int_0^T\int_{\mathbb{R}^2}\int_Z \nabla_yH_i(y)(u_0)^2\nabla_x\left(\frac{\partial \phi_0}{\partial x_i}(t,x)\right)dydxdt\nonumber\\
  &=2\sum_{i=1}^{2}\int_0^T\int_{\mathbb{R}^2}\int_Z \nabla_yH_i(y)u_0(\nabla_xu_0+\nabla_yu_1)\frac{\partial \phi_0}{\partial x_i}(t,x)dydxdt\nonumber\\
  &\hspace{3cm}{-2\sum_{i=1}^{2}\int_0^T\int_{\mathbb{R}^2}\int_Z \nabla_yH_i(y)u_0\nabla_xu_0\left(\frac{\partial \phi_0}{\partial x_i}(t,x)\right)dydxdt}\nonumber\\
  &={2}\sum_{i=1}^{2}\int_0^T\int_{\mathbb{R}^2}\int_Z \nabla_yH_i(y)u_0\nabla_yu_1\frac{\partial \phi_0}{\partial x_i}(t,x)dydxdt\nonumber\\
  &=-{2}\sum_{i=1}^{2}\int_0^T\int_{\mathbb{R}^2}\int_Z \Delta_yH_i(y)u_0u_1\frac{\partial \phi_0}{\partial x_i}(t,x)dydxdt\nonumber\\
  &={2}\int_0^T\int_{\mathbb{R}^2}\int_ZB(y)C(y)u_0u_1\nabla_x\phi_0(t,x)dydxdt\nonumber\\
  &=-{2}\int_0^T\int_{\mathbb{R}^2}\int_ZB(y)C(y)\nabla_x(u_0 u_1)\phi_0(t,x)dydxdt.\label{t7e19}
\end{align}
Using \ref{A3}, we get
\begin{align}
    \lim_{\es\rightarrow 0}\int_0^T\int_{\Omega^\es} f^\es(x,\dfrac{x}{\es}) \phi_0(t,x-\dfrac{B^*t}{\es}) dxdt= \int_0^T\int_{\mathbb{R}^2}\int_Z f(x,y)dydxdt.\label{t7e20}
\end{align}
Using \ref{A4} and Theorem \eqref{compactnessb},   gives
\begin{align}
    \lim_{\es\rightarrow 0}\es\int_0^T\int_{{\Gamma_N^\es}} g_N(u^\es) \phi_0(t,x-\dfrac{B^*t}{\es}) d\sigma dt&=\lim_{\es\rightarrow 0}\es\int_0^T\int_{{\Gamma_N^\es}} u^\es \phi_0(t,x-\dfrac{B^*t}{\es}) d\sigma dt\nonumber\\
    &=|\Gamma_N|\int_0^T\int_{\mathbb{R}^2}u_0(t,x)\phi_0(t,x)dxdt.\label{btwoscale}
\end{align}
By \eqref{t7e17}--\eqref{t7e20}, the passage to the limit $\es\rightarrow 0$ in  \eqref{t7e16} discovers the weak form
\begin{multline}\label{t7e21}
    |Z|\int_0^T\int_{\mathbb{R}^2}\partial_t u_0 \phi_0(t,x) dxdt+\int_0^T\int_{\mathbb{R}^2}\int_Z  D(y) \nabla_x u_0 \nabla_x \phi_0(t,x) dxdt\\
    +\int_0^T\int_{\mathbb{R}^2}\int_Z  D(y) \nabla_y u_1 \nabla_x \phi_0(t,x) dxdt-\int_0^T\int_{\mathbb{R}^2}\int_Z(B^*-B(y))\nabla_x u_1\phi_0(t,x)dydxdt\\
    -2\int_0^T\int_{\mathbb{R}^2}\int_ZB(y)C(y)\nabla_x(u_0 u_1)\phi_0(t,x)dydxdt=\int_0^T\int_{\mathbb{R}^2}\int_Z f(x,y)\phi_0(t,x)dydxdt\\
    -|\Gamma_N|\int_0^T\int_{\mathbb{R}^2}u_0(t,x)\phi_0(t,x)dxdt.
\end{multline}
Inserting the ansatz \eqref{cell} into \eqref{t7e21}, we can rewrite the last identity as
\begin{multline}\label{t7e23}
    \int_0^T\int_{\mathbb{R}^2}\partial_t u_0 \phi_0(t,x) dxdt+\int_0^T\int_{\mathbb{R}^2} D^*(u_0,W) \nabla_x u_0 \nabla_x \phi_0(t,x) dxdt\\
    =\frac{1}{|Z|}\int_0^T\int_{\mathbb{R}^2}\int_Z f(x,y)\phi_0(t,x)dydxdt-\frac{|\Gamma_N|}{|Z|}\int_0^T\int_{\mathbb{R}^2}u_0(t,x)\phi_0(t,x)dxdt,
\end{multline}
where $D^*(u_0,W)$ defined as \eqref{effectdiff} and \eqref{t7e23} is the weak formulation of \eqref{homeq1}. 
\end{proof}

    It is worth noting that the homogenization limit derived in Theorem \ref{Homlimit} and the corrector convergence result stated in Theorem \ref{Cor} in the next section are still valid if we replace the assumption $g_N(r)=r$ with $g_N(r)= k$, where $k\in \mathbb{R}$ is  fixed arbitrarily.  

\section{Searching for correctors}\label{searchforcor}
In this section, we study the strong convergence of the corrector term obtained from the homogenized limit problem. To prove such corrector-type result, we rely on techniques similar to those used in \cite{allaire2010homogenization}  to prove the strong convergence of solutions to the microscopic problem in $L^2$. We begin with stating  two auxiliary lemmas  which later will be employed to prove the wanted strong convergence result. We omit their proofs since they are straightforward  extensions of classical results related to the concept of two-scale convergence (Theorem 17 of \cite{lukkassen2002two}) and, respectively, to the two-scale convergence with drift (see in particular \cite{allaire2010homogenization}). 

\begin{lemma}\label{sem}
    Let $v^\es\in L^2(0,T;L^2(\Omega^\es)) $ such that $ v^\es \xrightharpoonupdown{$2-drift(B^*)$} v_0$ as $\es \rightarrow 0$ for some $v_0=v_0(t,x,y)\in L^2((0,T)\times \mathbb{R}^2;L_\#^2(Z))$. Then it holds  
    \begin{align}
        \liminf_{\es \rightarrow 0}\int_0^T\int_{\Omega^\es}(v^\es)^2 dxdt\geq \int_0^T\int_{\mathbb{R}^2}\int_Z v_0^2\; dydxdt.
    \end{align}
\end{lemma}
\begin{lemma}\label{semb}
    Let $v^\es\in L^2(0,T;\Gamma_N^\es) $ such that 
    \begin{align}
   \lim_{\es\rightarrow 0}\es \int_{(0,T)\times\Gamma_N^\es} v^\es(t,x)\phi(t,x-\dfrac{B^*t}{\es},\frac{x}{\es})dxdt= \int_{0}^{T}\int_{\mathbb{R}^2}\int_{\Gamma_N} v_0(t,x,y)\phi(t,x,y)dydxdt,
\end{align}
    for some $v_0(t,x,y)\in L^2((0,T)\times \mathbb{R}^2;L_\#^2(\Gamma_N))$.  Then it holds   
    \begin{align}
        \liminf_{\es \rightarrow 0}\es\int_0^T\int_{\Gamma_N^\es}(v^\es)^2 dxdt\geq \int_0^T\int_{\mathbb{R}^2}\int_{\Gamma_N} v_0^2\; dydxdt.
    \end{align}
\end{lemma}

\begin{lemma}\label{L9}
    Let $u^\es\in L^2(0,T;H^1(\Omega^\es))$ satisfy strongly two-scale convergence with drift to $u_0$ in the sense of Defenition \ref{defs2scd}. Then it holds: 
    \begin{align}
        \lim_{\es\rightarrow 0}\|u^\es\|_{L^2((0,T)\times\Omega^\es)}=\|u_0\|_{L^2((0,T)\times\mathbb{R}^2\times Z)}.
    \end{align}
\end{lemma}
\begin{proof}
By the Minkowski inequality, we have
\begin{align}
    \lim_{\es\rightarrow 0}\|u^\es(t,x)\|_{L^2((0,T)\times\Omega^\es)}\leq \lim_{\es\rightarrow 0}\|&u^\es(t,x)-u_0(t,x-\dfrac{B^*t}{\es},\dfrac{x}{\es})\|_{L^2((0,T)\times\Omega^\es)}\nonumber\\
    &+\lim_{\es\rightarrow 0}\|u_0(t,x-\dfrac{B^*t}{\es},\dfrac{x}{\es})\|_{L^2((0,T)\times\Omega^\es)}.\label{l9e1}
\end{align}
Now, using Lemma \ref{oscilationlemma} and \eqref{strongc} to deal with  \eqref{l9e1} leads to
\begin{align}
    0\leq \left|\lim_{\es\rightarrow 0}\|u^\es(t,x)\|_{L^2((0,T)\times\Omega^\es)}-\|u_0\|_{L^2((0,T)\times\mathbb{R}^2\times Z)}\right|\leq 0.\nonumber
\end{align}
\end{proof}
\par The main result of this section is stated in the next Theorem.
\begin{tcolorbox}
\begin{theorem}\label{Cor}
    Assume \ref{A1}--\ref{A5} hold, $g_N(r)=r $ for all $r\in\mathbb{R}$ {and $f^\es\xrightarrow{2-drift(B^*)} f$ }. Then  
\begin{equation}\label{corrector}
    \lim_{\es \rightarrow 0}\left\|\nabla\left( u^\es(t,\frac{x}{\es})-u_0(t,x-\dfrac{B^*t}{\es})-\es u_1(t,x-\dfrac{B^*t}{\es},\frac{x}{\es})\right)\right\|_{L^{2}(0,T;L^2(\Omega^\es))}=0,
\end{equation}
where $u^\es$ solves the original microscopic problem and $u_0$  and $u_1$ are given cf. \eqref{2drift1} and \eqref{2drift2}, respectively. 
\end{theorem}
\end{tcolorbox}

\begin{proof}
     We choose as test function $\phi=u^\es$ in the weak formulation \eqref{wf}. Integrating the result  from 0 to $t$, we obtain
    \begin{align}
   \int_0^t \int_{\Omega^\es} \partial_t u^\es u^\es dxds+\int_0^t \int_{\Omega^\es}D^\es\nabla u^\es\nabla u^\es dxds&- \frac{1}{\es} \int_0^t \int_{\Omega^\es}B^\es P(u^\es)\nabla u^\es dxds\nonumber \\
   &=  \int_0^t \int_{\Omega^\es}f^\es u^\es dxds-\es \int_0^t \int_{{\Gamma_N^\es}} g_N(u^\es) u^\es d\sigma ds.\label{ce1}
   \end{align}
   Using \eqref{zer} into \eqref{ce1}, yields
   \begin{align}
       \frac{1}{2} \int_{\Omega^\es} (u^\es)^2 dx +\int_0^t \int_{\Omega^\es}D^\es\nabla u^\es\nabla u^\es dxds=  \frac{1}{2} \int_{\Omega^\es} (u^\es(0))^2 dx&+\int_0^t \int_{\Omega^\es}f^\es u^\es dxds\nonumber\\
       &-\es \int_0^t \int_{{\Gamma_N^\es}} g_N(u^\es) u^\es d\sigma ds.\label{ce2}
   \end{align}
   {  Since $u^\es$ strongly converges to $u_0$, $f^\es$ strongly converges to $f$ in the sense of \eqref{strongc2}, we have $\lim_{\es\rightarrow 0}\int_0^t \int_{\Omega^\es}f^\es (x)u^\es (t,x)dxds=\int_0^t \int_{\mathbb{R}^2}\int_Z f(x,y) u_0(t,x)\; dydxds$}. 
   We integrate \eqref{ce2} from 0 to $p$, using Lemma \ref{L9}, Lemma \ref{oscilationlemma} and pass $\es\rightarrow 0$, we arrive at
 \begin{align}
       \frac{|Z|}{2} \int_0^p\int_{\mathbb{R}^2} (u_0)^2 dxdt+ \lim_{\es \rightarrow 0}&\int_0^p \int_0^t \int_{\Omega^\es}D^\es\nabla u^\es\nabla u^\es dxdsdt
       =  \frac{|Z|}{2} \int_0^p\int_{\mathbb{R}^2} g^2 dxdt
     \nonumber \\
     &+\int_0^p\int_0^t \int_{\mathbb{R}^2}\int_Z f u_0\; dydxdsdt - \lim_{\es \rightarrow 0}\es \int_0^p\int_0^t \int_{{\Gamma_N^\es}}  (u^\es)^2 d\sigma dsdt.\label{ce3}
   \end{align}
  Using Lemma \ref{semb} and \eqref{btwoscale} in \eqref{ce3}, we have
   \begin{align}
       \frac{1}{2} \int_0^p\int_{\mathbb{R}^2} (u_0)^2 dxdt+ \lim_{\es \rightarrow 0}&\frac{1}{|Z|}\int_0^p \int_0^t\int_{\Omega^\es}D^\es\nabla u^\es\nabla u^\es dxdsdt \leq \frac{1}{2}  \int_0^p\int_{\mathbb{R}^2} g^2 dxdt
       \nonumber\\
      &+\frac{1}{|Z|}\int_0^p\int_0^t \int_{\mathbb{R}^2}\int_Z f u_0\; dydxdsdt- \frac{|\Gamma_N|}{|Z|}\int_0^p\int_0^t\int_{\mathbb{R}^2} u_0^2\; dx dsdt.\label{ce4}
   \end{align}
   Now,  observing the structure of $D^*$  as it appears in \eqref{effectdiff}, we obtain for any $\xi\in \mathbb{R}^2$  that 
   \begin{align}
       D^*\xi \cdot\xi =\frac{1}{|Z|}\int_Z D(\xi+\nabla_y \sum w_i \xi_i)\cdot D(\xi+\nabla_y \sum w_i \xi_i).\label{ce5}
   \end{align}
   Using the structure of $u_1$ from \eqref{cell} and \eqref{ce5}, we have
   \begin{align}
       D^*\nabla_x u_0 \cdot\nabla_x u_0 =\frac{1}{|Z|}\int_Z D(\nabla_x u_0+\nabla_y u_1)\cdot (\nabla_x u_0+\nabla_y u_1).
   \end{align}
   Consider the weak form of \eqref{homeq1} and choose the test function $u_0$. We thus obtain
   \begin{multline}\label{ce7}
    \int_0^t\int_{\mathbb{R}^2}\partial_t u_0 u_0 dxds+\int_0^t\int_{\mathbb{R}^2} D^*\nabla_x u_0 \nabla_x u_0 dxds\\
    =\frac{1}{|Z|}\int_0^t\int_{\mathbb{R}^2}\int_Z fu_0dydxds-\frac{|\Gamma_N|}{|Z|}\int_0^t\int_{\mathbb{R}^2}u_0^2dxds,
\end{multline}
and hence, it holds as well that 
 \begin{multline}\label{ce8}
    \frac{1}{2}\int_{\mathbb{R}^2} u_0^2 dx+\frac{1}{|Z|}\int_0^t\int_{\mathbb{R}^2} \int_Z D(\nabla_x u_0+\nabla_y u_1)\cdot (\nabla_x u_0+\nabla_y u_1)dy dxds\\
    =\frac{1}{|Z|}\int_0^t\int_{\mathbb{R}^2}\int_Z fu_0\;dydxds-\frac{|\Gamma_N|}{|Z|}\int_0^t\int_{\mathbb{R}^2}u_0^2dxds+ \frac{1}{2}  \int_{\mathbb{R}^2} g^2 dx.
\end{multline}
Integrating \eqref{ce8} from 0 to $p$ and then comparing with \eqref{ce4}, we get
\begin{align}
      \frac{1}{2}& \int_0^p\int_{\mathbb{R}^2} u_0^2 dxdt+ {\lim_{\es \rightarrow 0}\frac{1}{|Z|}\int_0^p \int_0^t\int_{\Omega^\es}D^\es\nabla u^\es\nabla u^\es dxdsdt}\nonumber\\
       &\leq \frac{1}{2}\int_0^p\int_{\mathbb{R}^2} u_0^2 dxdt+\frac{1}{|Z|}\int_0^p\int_0^t\int_{\mathbb{R}^2} \int_Z D(\nabla_x u_0+\nabla_y u_1)\cdot (\nabla_x u_0+\nabla_y u_1)dy dxdsdt.
\end{align}
So,
\begin{align}
      {\lim_{\es \rightarrow 0} \int_0^p \int_0^t\int_{\Omega^\es}D^\es\nabla u^\es\nabla u^\es dxdsdt}
       &\leq \int_0^p \int_0^t\int_{\mathbb{R}^2} \int_Z D(\nabla_x u_0+\nabla_y u_1)\cdot (\nabla_x u_0+\nabla_y u_1)dy dxdsdt.\label{ce14}
\end{align}
From Lemma \eqref{sem} and \eqref{2drift2},  we have 
\begin{align}
    \int_0^p\int_0^t\int_{\mathbb{R}^2} \int_Z D(\nabla_x u_0+\nabla_y u_1)\cdot (\nabla_x u_0+\nabla_y u_1)dy dxdsdt \leq  {\lim_{\es \rightarrow 0} \int_0^p\int_0^t \int_{\Omega^\es}D^\es\nabla u^\es\nabla u^\es dxdsdt}\label{ce15}.
\end{align}
Comparing \eqref{ce14} and \eqref{ce15}, lead us to
\begin{align}
    \lim_{\es \rightarrow 0}\int_0^p\int_0^t\int_{\Omega^\es}D^\es\nabla u^\es\nabla u^\es dxdsdt=\int_0^p\int_0^t\int_{\mathbb{R}^2} \int_Z D(\nabla_x u_0+\nabla_y u_1)\cdot (\nabla_x u_0+\nabla_y u_1)dy dxdsdt.\label{ce9}
\end{align}
Now, differentiating \eqref{ce9} with respect to $p$, using the Fundamental Theorem of Calculus,  and finally choosing $t=T$, allows us to write
\begin{align}
    \lim_{\es \rightarrow 0}\int_0^T\int_{\Omega^\es}D^\es\nabla u^\es\nabla u^\es dxdt=\int_0^T\int_{\mathbb{R}^2} \int_Z D(\nabla_x u_0+\nabla_y u_1)\cdot (\nabla_x u_0+\nabla_y u_1)dy dxdt.\label{ce}
\end{align}
By the ellipticity condition of $D$, we have 
\begin{align}
    \theta &\left\|\nabla\left( u^\es(t,x)-u_0(t,x-\dfrac{B^*t}{\es})-\es u_1(t,x-\dfrac{B^*t}{\es},\frac{x}{\es})\right)\right\|_{L^{2}(0,T;L^2(\Omega^\es))} \nonumber\\
    &\leq \int_0^T \int_{\Omega^\es}D^\es \nabla\left( u^\es(t,x)-u_0(t,x-\dfrac{B^*t}{\es})-\es u_1(t,x-\dfrac{B^*t}{\es},\frac{x}{\es})\right) \nonumber\\
    &\hspace{1cm}\cdot \nabla\left( u^\es(t,x)-u_0(t,x-\dfrac{B^*t}{\es})-\es u_1(t,x-\dfrac{B^*t}{\es},\frac{x}{\es})\right) dxdt.\label{ce10}
\end{align}
Using the definition of two-scale convergence with drift, we have
\begin{align}
    \lim_{\es \rightarrow 0}\int_0^T \int_{\Omega^\es}D^\es\nabla u^\es\nabla u_0 dxdt&=\int_0^T\int_{\mathbb{R}^2} \int_Z D(\nabla_x u_0+\nabla_y u_1)\cdot\nabla_x u_0 dy dxdt,\label{ce11}\\
    \lim_{\es \rightarrow 0}\int_0^T \int_{\Omega^\es}\es D^\es\nabla u^\es\nabla u_1 dxdt&=\int_0^T\int_{\mathbb{R}^2} \int_Z D(\nabla_x u_0+\nabla_y u_1)\cdot\nabla_y u_1 dy dxdt,\\
    \lim_{\es \rightarrow 0}\int_0^T \int_{\Omega^\es} D^\es\nabla u_0\nabla u_0dxdt&=\int_0^T\int_{\mathbb{R}^2} \int_Z D\nabla_x u_0\nabla_x u_0 dy dxdt,\\
    \lim_{\es \rightarrow 0}\int_0^T \int_{\Omega^\es}\es^2 D^\es\nabla u_1\nabla u_1 dxdt&=\int_0^T\int_{\mathbb{R}^2} \int_Z D\nabla_y u_1\nabla_y u_1dy dxdt,\label{ce12}
\end{align}
Now, using \eqref{ce9}, \eqref{ce11}--\eqref{ce12}, we finally obtain
\begin{multline}\label{ce13}
   \lim_{\es \rightarrow 0}  \int_0^T \int_{\Omega^\es}D^\es \nabla\left( u^\es(t,x)-u_0(t,x-\dfrac{B^*t}{\es})-\es u_1(t,x-\dfrac{B^*t}{\es},\frac{x}{\es})\right)\\
    \hspace{1cm}\cdot \nabla\left( u^\es(t,x)-u_0(t,x-\dfrac{B^*t}{\es})-\es u_1(t,x-\dfrac{B^*t}{\es},\frac{x}{\es})\right) dxdt=0.
\end{multline}
We conclude our proof by using \eqref{ce13} and  
 \eqref{ce10} to obtain the desired result \eqref{corrector}.
\end{proof}

\section{Conclusion and outlook}\label{conclusion}

We rigorously derived a macroscopic equation for a reaction-diffusion problem with nonlinear drift posed in an unbounded perforated domain together with Robin type boundary data.  The main challenge for the homogenization was the presence of the nonlinear drift term which is scaled with a factor of order of  $\mathcal{O}\left(\dfrac{1}{\es}\right)$.  The key tools used to handle the homogenization asymptotics were the two-scale convergence with drift and the strong convergence formulated in moving co-ordinates. To study the well-posedness of microscopic problem we  relied on a classical result from \cite{ladyzhenskaia1988linear} that we  extended to cover the case of the unbounded perforated domain as needed in our target problem. To do so, we used a suitable  comparison principle jointly with a monotonicity argument.

 The obtained upscaled equation is a reaction-dispersion equation posed in an unbounded domain, strongly coupled with an elliptic cell problem posed in a bounded domain.  The corresponding dispersion tensor carries information concerning the microscopic diffusion and drift. We also provided a strong convergence result related to the structure of the corrector, exploiting the difference in the micro- and  macro-solutions in the $H^1$ norm. 

 We did perform the analysis and subsequent discussion in two-space dimensions since the original modeling of our equations was done in terms of  interacting particle systems  inter-playing in a plane. However, both the homogenization result and the corrector convergence hold  for higher dimensions.  One technical assumption deserves further attention -- we assumed here that the mean value of the coefficient in front of the nonlinear drift is zero. We believe it is a challenging open problem to handle the homogenization of such nonlinear drift case  when the mean value of this drift coefficient is non vanishing. 

 In this work, we only discussed about the convergence of the corrector. From a more practical perspective, it would be though very useful to derive a corrector estimate which can later be used to analyse the problem numerically. Till now, to our knowledge, the only corrector estimate result related to large-drift homogenization problems is reported in \cite{ouaki2012multiscale,ouaki2013etude}, where the authors considered a  linear diffusion-large convection problem. Along the same line of thinking, developing a two-scale finite element approach  for the micro-macro problem could be an excellent  research direction to pursue (following e.g. the works \cite{henning2010heterogeneous}, \cite{henning2011note}, or \cite{abdulle2017numerical}). The challenging parts for the numerical analysis are the presence of the large nonlinear drift occurring in the microscopic problem as well as  the strong coupling through the transport coefficient in the upscaled problem. 

\section*{Acknowledgments} We thank  H. Hutridurga (Bombay) for fruitful discussions during his KAAS seminar. V.R. thanks M. Eden (Karlstad) for interesting discussions related to the corrector result. 
The work of V.R. and A.M. is partially supported by the Swedish Research Council's project ``{\em  Homogenization and dimension reduction of thin heterogeneous layers}" (grant nr. VR 2018-03648).



\bibliographystyle{amsplain}
	\bibliography{main}

\end{document}